\documentclass[11pt]{amsart}
\usepackage{amsmath}
\usepackage{amsfonts}
\usepackage{amssymb}
\usepackage{amsthm}
\usepackage{latexsym,bm}
\usepackage{color}
\usepackage{graphicx}

\title[Ricci Curvature on Alexandrov spaces]{Ricci Curvature on Alexandrov spaces and Rigidity Theorems}
\author{Hui-Chun Zhang}
\address{Department of Mathematics\\  Sun Yat-sen University\\ Guangzhou 510275\\ E-mail address: zhhuich@mail2.sysu.edu.cn}
\author{Xi-Ping Zhu}
\address{Department of Mathematics\\  Sun Yat-sen University\\ Guangzhou 510275\\ E-mail address: stszxp@mail.sysu.edu.cn}
\newtheorem{thm}{Theorem}[section]
\newtheorem{prop}[thm]{Proposition}
\newtheorem{lem}[thm]{Lemma}

\newtheorem{cor}[thm]{Corollary}

\theoremstyle{definition}
\theoremstyle{remark}

\newtheorem{defn}[thm]{Definition}
\newtheorem{rem}[thm]{Remark}

\numberwithin{equation}{section}

\newcommand{\ua}{\uparrow}

\newcommand{\ls}{\leqslant}
\newcommand{\gs}{\geqslant}
\newcommand{\wa}{\widetilde\angle}
\newcommand{\wt}{\widetilde}

\newcommand{\ip}[2]{\left<{#1},{#2}\right>}

\newcommand{\be}[2]{\begin{#1}{#2}\end{#1}}
\newcommand{\R}{\mathbb{R}}

\begin{document}

\maketitle
\begin{abstract} In this paper, we introduce a new notion for lower  bounds of Ricci curvature on Alexandrov spaces,  and
extend Cheeger--Gromoll splitting theorem and Cheng's maximal diameter theorem to Alexandrov spaces under this Ricci curvature condition.
\end{abstract}
%\tableofcontents
%\setcounter{tocdepth}{1}
\section{Introduction}\label{intro}
Alexandrov spaces with curvature bounded below generalize
successfully the concept of lower  bounds  of sectional curvature
from Riemannian manifolds to singular spaces. The seminal paper [BGP]
and  the 10th chapter in the text book [BBI] provide excellent
introductions to this field. Many  important theorems in Riemannian
geometry had been extended to  Alexandrov spaces, such as  Synge's
theorem [Pet1], diameter sphere theorem [Per1], Toponogov splitting
theorem [Mi], etc.

However, many fundamental results in Riemannian geometry (for
example, Bishop--Gromov volume comparison theorem, Cheeger--Gromoll
splitting theorem and Cheng's maximal diameter theorem) assume only
the lower bounds on \emph{Ricci} curvature, not on sectional
curvature. Therefore, it is a very interesting question  how to
generalize the concept of lower  bounds of Ricci curvature from
Riemannian manifolds to singular spaces.

Perhaps the first concept of lower bounds of Ricci curvature on
singular spaces was given by Cheeger and Colding (see Appendix 2 in [CC2.I]).
They, in [CC1, CC2], studied Gromov--Hausdorff limit spaces of Riemannian manifolds
with Ricci curvature (uniformly) bounded below. Among other results
in [CC1], they proved the following rigidity theorem:
\begin{thm}  (Cheeger--Colding)

Let $M_i$ be a sequence of Riemannian manifolds and $M_i$ converges to $X$ in sense of Gromov--Hausdorff.\\
\indent (1)\  If $X$ contains a line and $Ric(M_i)\gs -\epsilon_i$ with $\epsilon_i\to0$, then $X$ is isometric to a direct product $\mathbb{R}\times Y$ over some length space $Y$.\\
\indent (2)\  If $Ric(M_i)\gs n-1$ and diameter of $M_i$ $diam(M_i)\to\pi,$  then $X$ is isometric to a spherical suspension $[0,\pi]\times_{\sin}Y$ over some length space $Y$.
\end{thm}

In [Pet4], Petrunin considered to generalize the lower bounds of Ricci curvature for singular spaces via subharmonic functions.

Recently, in terms of $L^2-$Wasserstein space and optimal mass transportation,  Sturm [S1, S2] and Lott--Villani  [LV1, LV2] have given a
 generalization of ``Ricci curvature has lower bounds" for metric measure spaces\footnote{A metric measure space is a metric space equipped a Borel measure.},
 independently. They call that curvature-dimension conditions, denoted by $CD(n,k)$ with  $n\in(1,\infty]$ and $k\in\mathbb{R}.$ For the convenience of readers,
 we repeat their definition of $CD(n,k)$ in the Appendix of this paper. On the other hand, Sturm in [S2] and Ohta in [O1] introduced  another definition of ``Ricci curvature bounded below" for metric measure spaces, the measure contraction property $MCP(n,k)$, which is a slight modification of a property introduced earlier by Sturm in [S3] and in a similar form by Kuwae and Shioya in [KS3, KS4].
 The condition $MCP(n,k)$  is indeed an infinitesimal version of the Bishop--Gromov relative volume comparison condition.
 For a  metric measure space, Sturm [S2] proved that $CD(n,k)$ implies $MCP(n,k)$ provided it is
 non-branching\footnote{A geodesic space is called non-branching if  for any quadruple points $z, x_0,x_1,x_2$ with $z$  being the midpoint of $x_0$ and $x_1$ as well as the midpoint of $x_0$ and $x_2$ , it follows that  $x_1=x_2$.}.
 Note that any Alexandrov space with curvature bounded below is non-branching. Recently, Petrunin [Pet2]
proved that any $n$-dimensional Alexandrov space with curvature $\gs
0$ must satisfy $CD(n,0)$ and claimed  the general statement that the condition curvature $\gs k$ (for some $k\in \R$) implies the condition $CD(n,(n-1)k) $ can be also proved along the same lines.

Let $M$ be a Riemannian manifold with Riemannian distance $d$ and Riemannian volume $vol$.  Lott, Villani in [LV1] and von Renesse, Sturm in [RS, S4] proved that $(M,d,vol)$ satisfies $CD(\infty,k)$ if and only if  $Ric(M)\gs k$. Indeed, they proved a stronger weighted version (see Theorem 7.3 in [LV1] and Theorem 1.1 in [RS], Theorem 1.3 in [S4]). Let $\phi$ be a smooth function on $M$ with $\int_Me^{-\phi}dvol=1$. Lott and Villani in [LV2] proved that $(M,d, e^{-\phi}\cdot vol)$ satisfies $CD(n,k)$ if and only if weighted Ricci curvature $Ric_n(M)\gs k$ (see Definition 4.20-- the definition of $Ric_n$-- and Theorem 4.22 in [LV2]). A similar result was proved by Sturm in [S2] (see Theorem 1.7 in [S2]). In particular, they proved that  $(M,d, vol)$ satisfies $CD(n,k)$ if and only if $Ric(M)\gs k$ and  $\dim(M)\ls n.$ If  $\dim(M)= n,$
  Ohta in [O1] and Sturm  in [S2] proved, independently, that $M$ satisfies $MCP(n,k)$
is equivalent to $Ric(M)\gs k$.

Nevertheless, since $n$-dimensional norm spaces $(V^n,\|\cdot\|_p)$ satisfy $CD(n,0)$ for every $p>1$ (see, for example, page 892 in [V]),
 it is impossible to show Cheeger-Gromoll splitting theorem under $CD(n,0)$ for general metric measure spaces. Furthermore, it was shown by
Ohta  in [O3] that on a Finsler manifolds $M$, the
curvature-dimension condition $CD(n,k)$ is equivalent to the weighted
Finsler Ricci curvature condition $Ric_n(M)\gs k$ (see also [O4] or  [OSt], refer to [O4] for the definition $Ric_n$ in Finsler manifolds). That says, the
curvature-dimension condition is somewhat a Finsler geometry
character. Seemly, it is difficult to
 show the rigidity theorems, such as Cheng's maximal diameter theorem and Obata's theorem, under $CD(n,n-1)$ for general metric measure spaces.

As a compensation, Watanabe [W] proved that if a metric measure space $M$ satisfies $CD(n,0)$ or $MCP(n,0)$ then $M$ has at most two ends.  Ohta  [O2] proved  that a non-branching compact metric measure space with $MCP(n,n-1)$ and diameter $=\pi$  is homeomorphic to a spherical suspension.

Alexandrov spaces with curvature bounded below have richer geometric
information than general metric measure spaces. In particular, a
finite dimensional norm space with curvature bounded below must be
an inner-product space. Naturally, one would expect that
Cheeger--Gromoll splitting theorem still holds on Alexandrov spaces
with suitable nonnegative ``Ricci curvature condition".

Recently in [KS1], Kuwae and Shioya  proved the following topological splitting theorem for Alexandrov spaces under the $MCP(n,0)$ condition:
\begin{thm} (Kuwae--Shioya)

 Let $M^n$ be an $n$-dimensional Alexandrov space. Assume that  $M^n$ contains a line.\\
\indent (1)\  If $M$ satisfies $MCP(n,0)$, then $M^n$ is homeomorphic to a direct product space $\mathbb{R}\times Y$ over some topological space $Y$.\\
 \indent(2)\  If the singular set of $M^n$ is closed and the non-singular set is an (incomplete) $C^\infty$ Riemannian manifold of $Ric\gs0$, then
$M^n$ is isometric to a direct product space $\mathbb{R}\times Y$ over some Alexandrov space $Y$.\end{thm}

We remark that Kuwae and Shioya actually  obtained  a  more general weighted measure version of the above theorem  in [KS2].

 In the following, inspired by  Petrunin's second variation of arc length [Pet1], we will introduce a new notion of the Ricci curvature bounded below for Alexandrov spaces.

Let $M$ be  an $n-$dimensioal Alexandrov space of curvature bounded from below \emph{locally} without boundary.  It is well known in [PP] or [Pet3] that, for any $p\in M$ and $\xi\in \Sigma_p$,  there exists a \emph{quasi-geodesic} starting at $p$ along direction $\xi.$ (See [PP] or section 5 in [Pet3] for the definition and properties of quasi-geodesics.) According to [Pet1],
the exponential map
$\exp_p: T_p\rightarrow M$ is defined as follows.  For any  $v\in T_p$,  $\exp_p(v)$ is a point on some  quasi-geodesic of length $|v|$ starting point $p$ along $v/|v|\in \Sigma_p$. If the quasi-geodesic is not unique, we take one of them as the definition of $\exp_p(v).$

 Let $\gamma:\ [0,\ell)\rightarrow M$ be a geodesic. Without loss of generality, we may assume that a neighborhood  $U_\gamma$ of $\gamma$ has curvature $\gs k_0$ for some $k_0<0$.

 According to Section 7 in [BGP], the tangent cone $T_{\gamma(t)}$ at an interior point $\gamma(t) \ (t \in (0,\ell))$ can be split  into a direct metric product. We denote \begin{equation*}\begin{split}L_{\gamma(t)}&=\{\xi\in T_{\gamma(t)}\ |\ \angle(\xi,\gamma^+(t))=\angle(\xi,\gamma^-(t))=\pi/2\},\\
\Lambda_{\gamma(t)}&=\{\xi\in \Sigma_{\gamma(t)}\ |\
\angle(\xi,\gamma^+(t))=\angle(\xi,\gamma^-(t))=\pi/2\}.
\end{split}\end{equation*}

In [Pet1], Petrunin proved the following second  variation formula of arc-length.
\begin{prop}(Petrunin)

Given any two points $q_1,q_2\in\gamma$, which are not end points,
and any positive number sequence $\{\varepsilon_j\}^\infty_{j=1}$
with $\varepsilon_j\to 0$, there exists a subsequence
$\{\wt\varepsilon_j\}\subset \{\varepsilon_j\}$ and an isometry $T:\
L_{q_1}\rightarrow\ L_{q_2}$ such that
\begin{align*}|\exp_{q_1}(\wt\varepsilon_j u),\ \exp_{q_2}(\wt\varepsilon_j Tv)|\ls & |q_1q_2|+ \frac{|uv|^2}{2|q_1q_2|}\cdot\wt\varepsilon^2_j\\&-\frac{k_0\cdot|q_1q_2|}{6}\cdot\big(|u|^2+|v|^2+\langle u,v\rangle\big)\cdot\wt\varepsilon_j^2+o(\wt\varepsilon_j^2)\end{align*}for any $u,v\in L_{q_1}.$
\end{prop}
We remark that for a $2-$dimensional Alexandrov space, Cao, Dai and Mei in [CDM] improved the second variation formula such that the above inequality holds for all $\{\varepsilon_j\}_{j=1}^\infty.$ But for higher dimensions, to the best of our knowledge, we don't know whether the parallel translation $T$ in the above second variation formula can be chosen independent of the sequences $\{\varepsilon_j\}.$

Based on this second variation formula, we can propose a condition which resembles the lower bounds for the radial curvature along the geodesic $\gamma$.

  Let $\{g_{\gamma(t)}\}_{0<t<\ell}$ be a family of functions, where for each $t$, $g_{\gamma(t)}$ is a continuous function on $\Lambda_{\gamma(t)}$. For simplicity, we call  $\{g_{\gamma(t)}\}_{0<t<\ell}$ to be a \emph{continuous function family}.
\begin{defn}
A continuous  function family  $\{g_{\gamma(t)}\}_{0<t<\ell}$ is said to satisfy \emph{Condition (RC)}, if
for any $\epsilon>0$ and any $t_0\in(0,\ell)$, there exists a neighborhood  $I_{t_0}:=(t_0-\tau^*,t_0+\tau^*)\subset(0,\ell)$ with the following property.
For any two number $s,t\in I_{t_0}$ with $s<t$ and for any sequence $\{\theta_j\}_{j=1}^\infty$ with $\theta_j\to 0$ as $j\to\infty$, there exists an isometry $T: \Lambda_{\gamma(t)}\rightarrow \Lambda_{\gamma(s)}$ and a subsequence $\{\delta_j\}$ of $\{\theta_j\}$ such that
\begin{equation}\begin{split} |\exp&_{\gamma(s)}(\delta_j l_1T\xi),\ \exp_{\gamma(t)}(\delta_j l_2\xi)|\\ \ls  &|s-t|+\frac{(l_1-l_2)^2}{2|s-t|}\cdot\delta_j^2\\
&-\frac{\big(g_{\gamma(t)}(\xi)-\epsilon\big)\cdot|s-t|}{6}\cdot\big(l_1^2+l_1\cdot l_2+l^2_2\big)\cdot\delta^2_j+o(\delta_j^2)\end{split}\end{equation} for any $l_1,l_2\gs0$ and any $\xi\in \Lambda_{\gamma(t)}$.
\end{defn}

Let  $\mathcal{F}$ denote the set all of continuous  function
families$\{g_{\gamma(t)}\}_{0< t<\ell}$,  which satisfy
Condition $(RC)$.

Clearly, the above proposition shows that  $\{g_{\gamma(t)}=k_0\}_{0<t<\ell}\in \mathcal{F}$.
\begin{defn} We say that $M$ has \emph{Ricci curvature bounded below by $ (n-1)K$ along $\gamma$},
if \begin{equation}\rho:=\sup_{\{g_{\gamma(t)}\}\in\mathcal F}\inf_{0<t<\ell }\oint_{\Lambda_{\gamma(t)}}g_{\gamma(t)}(\xi)\gs K,\end{equation}
where $\oint_{\Lambda_x}g_x(\xi)=\frac{1}{vol(\Lambda_x)}\int_{\Lambda_x}g_x(\xi)d\xi$.

 We say $M$ has \emph{ Ricci curvature  bounded below by $(n-1)K$ on an open set $U\subset M$}, if for each point $p\in U$, there is a neighborhood $U_p$ of $p$ with $U_p\subset U$ such that  $M$ has Ricci curvature bounded below by $ (n-1)K$ along every geodesic $\gamma:[0,\ell)\rightarrow U_p$.
 When $U=M$, we say $M$ has Ricci curvature bounded below by $(n-1)K$ and denote $Ric(M)\gs(n-1)K$.
 \end{defn}
\begin{rem} (i)  When $M$ is a smooth Riemannian manifold, by the second variation of formula of arc-length, it is easy to see Condition $(RC)$
is equivalent to $$sec_M(\Pi_t)\gs g_{\gamma(t)}(\xi),$$ where $\Pi_t\subset T_{\gamma(t)}$ is any $2-$dimensional subspace,  spanned  by $\gamma'(t)$ and a $\xi\in \Lambda_{\gamma(t)}$. Thus in a Riemannian manifold, our definition on Ricci curvature bounded below by $(n-1)K$ is exactly the  classical one.\\
\indent (ii) Let $M$ be an $n$-dimensional Alexandrov space  with curvature $\gs K$. The above Proposition 1.3 shows that  $Ric(M)\gs(n-1)K$.\\
\indent (iii) Recall that Petrunin in [Pet2] proved any  $n$-dimensional Alexandrov space $M$ with curvature $\gs K$ must satisfy the curvature-dimension condition $CD(n,(n-1)K)$. In the appendix, by modifying Petrunin's proof in [Pet2], we will show that any $n$-dimensional Alexandrov space $M$ with $Ric(M)\gs (n-1)K$ also satisfies $CD(n,(n-1)K)$.\\
 \indent (iv) At the present stage, we don't know  if the Ricci curvature condition $Ric(M)\gs(n-1)K$ is equivalent to the curvature-dimension condition $CD(n,(n-1)K).$ We will investigate this question in future.
\end{rem}

Our main results in this paper are the following  splitting theorem and  maximal diameter theorem.
\begin{thm} (Splitting theorem)

Let $M$ be an $n$-dimensional complete non-compact Alexandrov space
with  nonnegative Ricci curvature and $\partial M=\varnothing$. If
$M$ contains a line, then  $M$ is isometric to a direct metric
product $\mathbb{R}\times N$ for some Alexandrov space $N$ with nonnegative Ricci curvature.
\end{thm}
 \begin{thm}(Maximal diameter theorem)

   Let $M$ be an $n$-dimensional compact Alexandrov space with  Ricci curvature bounded below by $n-1$ and $\partial M=\varnothing$. If the diameter of $M$ is $\pi$, then $M$ is isometric to a spherical suspension  over  an Alexandrov space with curvature $\gs1$.
\end{thm}
 An open question for the curvature-dimension condition $CD(n,k)$($k\not=0$) is ``from local to global" (See, for example, the 30th chapter in [V]).
 In particular, given a metric measure  space which admits a covering and   satisfies $CD(n,k)$ ($k\not=0$), we don't know if  the covering space with pullback  metric still satisfies  $CD(n,k)$.

 One advantage of our definition of the Ricci curvature bounded below on Alexandrov spaces is that the definition is purely local.
 In particular, any covering space of an $n$-dimensional Alexandrov space with Ricci curvature bounded below by $(n-1)K$ still satisfies the condition $Ric\gs(n-1)K.$
 Meanwhile, we note that  Bishop--Gromov volume comparison theorem also holds on an Alexandrov space with Ricci curvature bounded below
 (see Corollary A.3 in Appendix). Consequently, the same proofs as in Riemannian manifold case (see [A] and, for example, page 275-276 in [P]) give the
 following estimates on the fundamental group and the first Betti number.

\begin{cor}Let $M$ be a compact $n$-dimensional Alexandrov space with nonnegative Ricci curvature and $\partial M=\varnothing$. Then its fundamental group has a finite index Bieberbach subgroup.
\end{cor}

\begin{cor}Let $M$ be an $n$-dimensional  Alexandrov space with nonnegative Ricci curvature and $\partial M=\varnothing$.  Then any finitely generated subgroup of $\pi_1(M)$ has polynomial growth of degree $\ls n$. If some finitely generated subgroup of $\pi_1(M)$ has polynomial growth of degree $=n$, then $M$ is compact and flat.\end{cor}

\begin{cor}
   Let $M$ be an $n$-dimensional  Alexandrov space with  $\partial M=\varnothing$. \\
   \indent (1) If  $Ric(M)\gs(n-1)K>0$, then its fundamental group is finite.\\
   \indent (2)   If $Ric(M)\gs(n-1)K$ and diameter of $M\ls D$, then   $$b_1(M)\ls C(n, K^2\cdot D)$$ for some function $C(n, K^2\cdot D)$.

    Moreover, there exists a constants $\kappa(n)>0$ such that if $K^2\cdot D\gs -\kappa(n),$ then $b_1(M)\ls n.$
\end{cor}

The paper is organized as follows. In Section 2, we recall some
necessary materials for Alexandrov spaces. In Section 3,  we will
define a new representation of Laplacian along a geodesic and will prove the
comparison theorem for the newly-defined representation of Laplacian (see Theorem 3.3).
In Section 4, we will discuss the rigidity part of the comparison
theorem. The maximal diameter theorem and the splitting theorem will
be proved in Section 5 and 6, respectively. In the appendix, we give
a modification of Petrunin's proof in [Pet2] to show that the
condition on Ricci curvature bounded below implies the
curvature-dimension condition (see Proposition A.2).

{\bf Acknowledgements}\ \      We would like to thank Dr. Qintao
Deng for helpful discussions.  We are also grateful to the referee for helpful comments on the second variation formula. The second author is partially
supported by NSFC 10831008 and NKBRPC 2006CB805905.

\section{Preliminaries}
 A metric space  $(X,|\cdot,\cdot|)$ is called a length space if for any two point $p,q\in X$, the distance between $p$ and $q$ is given by
  $$|pq|=\inf_{\gamma,\gamma\ {\rm connect}\ p,q}Length(\gamma).$$
A length space $X$ is called a \emph{geodesic} space if for any two
point $p,q\in X$, there exists a curve $\gamma$ connecting $p$ and
$q$ such that $Length(\gamma)=|pq|$. Such a curve is called a
shortest curve. A geodesic is a  \emph{unit-speed} shortest curve.

Recall that a length space $X$ has curvature $\gs k$ in an open set
$U\subset X$ if for  any quadruple $(p;a, b,c)\subset U $, there
holds  $$\wa_k apb+\wa_k bpc+\wa_k cpa\ls 2\pi,$$ where $\wa_k apb,
\wa_k bpc,$ and $\wa_k cpa$ are the comparison angles in the
$k-$plane. A length space $M$ is called an\emph{ Alexandrov space
with curvature bounded from below locally} (for short, we say $M$ to
be an \emph{Alexandrov space}), if it is locally compact and any
point in $M$ has an open neighborhood $U\subset M$ such that $M$ has
curvature  $\gs k_U$ in $U$, for some $k_U\in \mathbb{R}$.

Let $M$ be an Alexandrov space without boundary and $U\subset M$ be
an open set. A locally Lipschitz function $u$ on $U$ is said to be
\emph{$\lambda-$concave}  on $U$ if for any geodesic $\gamma\subset
U$, the one-variable function $$u\circ\gamma(t)-\lambda t^2/2$$ is
concave. A function $u$ on $M$ is said to be \emph{semi-concave} if
for any point $x\in M$  there is a neighborhood $U_x\ni x $ and a
real number $\lambda_x$ such that the restriction $u|_{U_x}$ is
$\lambda_x$-concave.

Let $\psi: \mathbb{R}\rightarrow \mathbb{R}$  be a continuous function. A function $u$ on $M$ called $\psi(u)-concave$ if for any point $x \in M$ and any $\varepsilon > 0$ there is a neighborhood $U_x\ni x $ such that $u|_{U_x}$ is $(\psi\circ u(x) +\varepsilon)$-concave.

 If $M$ has curvature $\gs k$ in $U$, then it is well-known that the function $u=\varrho_k\circ dist_p$ is $(1-k u)-$concave in $U\backslash\{p\}$,
 where  $$\varrho_k(\upsilon)=\begin{cases}\frac{1}{k}\big(1-\cos(\sqrt k\upsilon)\big)\ & {\rm if}\quad k>0,\\
\frac{\upsilon^2}{2}\ &{\rm if}\quad k=0,\\
\frac{1}{k}\big(\cosh(\sqrt{-k}\upsilon)-1\big) \ & {\rm if}\quad
k<0,
\end{cases} $$
(see, for example, Section 1 in [Pet3]).

Let $u$ be a semi-concave function on $M$. For any point $p\in M$,
there exists a \emph{$u-$gradient curve} starting at $p$. Hence $u$
generates a \emph{gradient flow} $\Phi_u^t: M\rightarrow M$, which
is a locally Lipschitz map. (Actually, it is just a semi-flow, because backward flow $\Phi_u^{-t}$ is not always well-defined.) Particularly, if $u$ is concave, the
gradient flow is a 1-Lipschitz map. We refer to Section 1 and 2 in
[Pet3] for the details on semi-concave functions, gradient curves
and gradient flows.

\section{Laplacian comparison theorem}
Let $M$ be an $n$-dimensional Alexandrov space without boundary.  A canonical Dirichlet form $\mathcal{E}$ is defined by $$\mathcal{E}(u,v):=\int_M\ip{\nabla u}{\nabla v}d{\rm vol},\qquad {\rm for}\ \ u,v\in W^{1,2}_0(M).$$(see [KMS]). The Laplacian associated to the canonical Dirichlet  form is given as follows.
Let $u: U\subset M\to \R$ be a $\lambda-$concave function. The (canonical) Lapliacian of $u$ as a sign-Radon measure is defined by
 \begin{equation*}\int_M \phi d\Delta u=-\mathcal{E}(u,\phi)=-\int_M\ip{\nabla \phi}{\nabla u}d{\rm vol}\end{equation*} for all Lipschitz function $\phi$ with compact support in $U.$ In [Pet2], Petrunin  proved \begin{equation*}\Delta u\ls n\lambda\cdot{\rm vol},\end{equation*} in particular, the singular part of $\Delta u$ is non-positive. If $M$ has curvature $\gs K$, then any distance function $dist_p(x):=d(p,x)$ is $\cot_K\circ dist_p-$concave on $M\backslash\{p\}$,  where the function $\cot_K(s)$ is defined  by $$\cot_K(s)=\begin{cases}\frac{\sqrt K\cdot\cos(\sqrt K s)}{\sin(\sqrt Ks)}\ & {\rm if}\quad K>0,\\
\frac{1}{s}\ &{\rm if}\quad K=0,\\
\frac{\sqrt{-K}\cdot\cosh(\sqrt{-K}s)}{\sinh(\sqrt{-K}s)}\ & {\rm if}\quad K<0.
\end{cases} $$It is a solution of the ordinary differential equation $\chi'(s)=-K-\chi^2(s)$. Therefore the above inequality $\Delta u\ls n\lambda\cdot{\rm vol}$ gives a Laplacian comparison theorem for the distance function on Alexandrov spaces.

   In [KS1], by using the $DC-$structure (see [Per2]),  Kuwae--Shioya
defined a distributional Laplacian for a distance function $dist_p$ by $$\Delta dist_p=D_i\big(\sqrt{det(g_{ij})}g^{ij}\partial_j dist_p\big)$$ on a local chart of $M\backslash \mathcal S_\epsilon$ for sufficiently small  positive number $\epsilon$, where $$\mathcal S_\epsilon:=\{x\in M:\ {\rm vol}(\Sigma_x)\ls {\rm vol}(\mathbb{S}^{n-1})-\epsilon\}$$and $D_i$ is the distributional derivative. Note that the union of all $\mathcal S_\epsilon$ has zero measure. One can view the distributional Laplacian $\Delta dist_p$ as a sign-Radon measure. In [KMS], Kuwae, Machigashira and Shioya proved that  the distributional Laplacian is  actually a representation of  the previous (canonical) Laplacian on $M\backslash \mathcal S_\epsilon$.  Moreover in [KS1], Kuwae and Shioya extended the Laplacian comparison theorem  under
the weaker condition $BG(k)$.
%\begin{thm} $($\cite{KS-lap}$)$\indent If an $n$-dimensional Alexandrov space $ M$ satisfies $BG(k)$,
%then $$ d \Delta dist_p \ls (n -1) \cot_k\circ dist_p\cdot d{\rm vol}\qquad {\rm on}\quad M\backslash(\{p\}\cup S_\epsilon).$$
%\end{thm}

 Both of the above canonical Laplacian  and its DC representation (i.e. the distributional Laplacian) make sense up to a set
which has zero measure.

In Riemannian geometry, according to Calabi, the Laplacian comparison theorem holds in barrier sense, not just in distribution sense. In this section, we will try to give a new representation of  the above canonical Laplacian of a distance
function, which makes sense in $W_p$, the set of points $z\in M$ such that the geodesic $pz$ can be extend beyond $z$. We will also prove a  comparison theorem for the new representation under our Ricci curvature condition.

 Let $M$ denote an $n$-dimensional complete Alexandrov space without boundary.  Fix a geodesic $\gamma:[0,\ell)\rightarrow M$ with  $\gamma(0)=p$ and denote $f=dist_{p}$. Let $x\in \gamma\backslash\{p\}$ and $L_x,\ \Lambda_x$ be as above in Section 1. Clearly, we  may assume  that $M$ has curvature $\gs k_0$ (for some $k_0<0$)  in a neighborhood  $U_\gamma$ of $\gamma$.

Perelman in [Per2] defined a Hessian for a semi-concave function $u$ on almost all point $x\in M$, denoted by $Hess_xu$. It is a bi-linear form on $T_x\ (=\R^n)$. But for the given geodesic $\gamma$, we can not insure that the Hessian is well defined along $\gamma$.

We now define a  version of Hessian and Laplacian for the distance function $f$ along the geodesic $\gamma$ as follows. Note that the tangent space at an interior point $x\in \gamma$ can be split to $L_x\times\R$ and $f\circ\gamma$ is linear. So we only need to define the Hessian  on the set of orthogonal directions $\Lambda_x$.

Throughout this paper,  $\mathcal S$ will always denote the set of
all sequences $\{\theta_j\}^\infty_{j=1}$ with $\theta_j \to0$ as
$j\to\infty$ and $\theta_{j+1}\ls \theta_j$.

\begin{defn}Let  $x\in \gamma\backslash\{p\}$.  Given a sequence $\theta:=\{\theta_j\}_{j=1}^\infty\in\mathcal S$, we define a function $H^\theta_xf:\Lambda_x\rightarrow\mathbb{R}$ by
$$H^\theta_xf(\xi)\overset{def}{=}\limsup_{s\to0, \ s\in \theta}\frac{f\circ\exp_x(s\cdot\xi)-f(x)}{s^2/2};$$
and $$\Delta^\theta f(x)\overset{def}{=}(n-1)\cdot\oint_{\Lambda_x} H^\theta_xf(\xi).$$
\end{defn}

Since $U_\gamma$ has curvature $\gs k_0$, we know that  $f$ is $\cot_{k_0}(|px|)-$concave  and $dist_{\gamma(\ell)}$ is $\cot_{k_0}(|x\gamma(\ell)|)-$concave near $x$, which imply  \begin{equation}H^\theta_xf\ls \cot_{k_0}(|px|)\end{equation} for any sequence $\theta\in\mathcal S$, and $$ |\gamma(\ell)\ \exp_x(s\cdot\xi)|\ls |x\gamma(\ell)|+\cot_{k_0}(|x\gamma(\ell)|)\cdot s^2/2+o(s^2)$$ for any $\xi\in\Lambda_x$. Then by triangle inequality, we have  \begin{equation}H^\theta_xf\gs -\cot_{k_0}(|x\gamma(\ell)|).\end{equation} Thus $H^\theta_xf$ is well defined and bounded.
It is easy to see that $H_x^\theta f$ is measurable on $\Lambda_x$ and thus it is  integrable.

 If there exists Perelman's Hessian of $f$ at a point  $x$ (see [Per2]), then $H_x^\theta f(\xi)=Hess_xf(\xi,\xi)$ for all $\xi\in \Lambda_x$ and $\theta\in\mathcal S$.

Denote by $Reg_f$  the set of points $z\in M$ such that there exists Perelman's Hessian of $f$ at  $z$.
If we write  the Lebesgue decomposition of the canonical Laplacian $\Delta f=(\Delta f)^{sing}+(\Delta f)^{ac}\cdot{\rm vol}$,  with respect to the $n$-dimension Hausdorff measure ${\rm vol}$, then $(\Delta f)^{ac}(x)={\rm Tr}Hess_xf=\Delta^\theta f(x)$ for all $x\in W_p\cap Reg_f$ and $\theta\in\mathcal S$. It was shown  in [OS, Per2] that $Reg_f\cap W_p$ has full measure in $M$. Thus  $\Delta^\theta f(x)$ is actually a representation of the absolutely continuous part of the canonical Laplacian $\Delta f$ on $W_p$.

Note from the definition that if $\theta_1\subset\theta_2$, then
$$H_x^{\theta_1}f\ls H_x^{\theta_2}f \mbox{\ \ and \ \ } \Delta^{\theta_1}f(x)\ls
\Delta^{\theta_2}f(x).$$
 The following lemma is a discrete  version of the propagation equation of the Hessian of $f$ along the geodesic $\gamma$.
\begin{lem} Let $f=dist_p$. Given $\epsilon>0$, a continuous functions family  $\{g_{\gamma(t)}\}_{0<t<\ell}$ and a sequence $\{\theta_j\}_{j=1}^\infty\in \mathcal S$.  Let $y,z\in \gamma$ with $|py|<|pz|.$ We assume that a isometry $T:\Lambda_z\to\Lambda_y$ and  the subsequence $\delta:=\{\delta_j\}\subset\{\theta_j\} $ such that (1.1) holds. Then   \begin{equation}H^\delta_zf(\xi)\leqslant l^2\cdot H^\delta_yf(\eta)+ \frac{(l-1)^2}{|yz|}-\frac{l^2+l+1}{3}\cdot|yz|\cdot\big(g_z(\xi)-\epsilon\big) \end{equation}for any $l\geqslant0$ and any $\xi\in\Lambda_z,\ \eta=T\xi\in\Lambda_y$.
\end{lem}
\begin{proof}
  For any $\xi\in \Lambda_z,$ we can choose a subsequence $\{\delta'_j\}\subset \{\delta_j\}$ such that $$H^\delta_zf(\xi)=\lim_{j\to\infty}\frac{f(\exp_z(\delta'_j\xi))-f(z)}{\delta_j'^{2}/2}.$$
Then, we have \begin{equation}f(\exp_z(\delta'_j\xi))=f(z)+\frac{\delta_j'^2}{2}H^\delta_zf(\xi) +o(\delta_j'^2)\end{equation} for any $l\geqslant 0$. By definition, we have
\begin{equation}f(\exp_y(\delta'_j\cdot l\eta))\leqslant f(y)+\frac{(l\delta_j')^2}{2}H^\delta_yf(\eta) +o(\delta_j'^2).\end{equation}
Note that \begin{equation} f(z)-f(y)=|yz|\quad\end{equation} and\begin{equation} f(\exp_z(\delta'_j\xi))-f(\exp_y(\delta'_j\cdot l\eta)) \leqslant|\exp_z(\delta'_j\xi),\exp_y(\delta'_j\cdot l\eta)|.\end{equation}
By combining (3.4)--(3.7) and using (1.1) with $l_1=l, l_2=1$, we have\begin{align*}\frac{\delta_j'^2}{2}\big(H^\delta_z&f(\xi)-l^2\cdot H^\delta_yf(\eta)\big)+o(\delta_j'^2)\\
&\leqslant
\delta_j'^2\cdot\Big(\frac{(l-1)^2}{2|yz|}-\frac{g_r(\xi)-\epsilon}{6}\cdot|yz|\cdot(l^2+l+1)\Big)+o(\delta_j'^2)
,\end{align*}for any $l\geqslant0$. Hence
$$H^\delta_zf(\xi)-l^2\cdot H^\delta_yf(\eta)\leqslant
\frac{(l-1)^2}{|yz|}-\frac{l^2+l+1}{3}\cdot|yz|\cdot
(g_z-\epsilon).$$ This completes the proof of the lemma.\end{proof}
The following result is the comparison for the above defined representation of
Laplacian.
\begin{thm}Let  $f=dist_p$ and $x\in \gamma\backslash\{p\}$. If  $M$ has Ricci $\geqslant (n-1)K$ along the geodesic $\gamma(t)$, then,
given  any sequence $\{\theta_j\}_{j=1}^\infty\in\mathcal S$, there exists a subsequence $\delta=\{\delta_j\}$ of $\{\theta_j\} $
 such that $$\Delta^{\delta} f(x)\ls(n-1)\cdot\cot_K(|px|).$$
(If $K>0$, we add assumption $|px|<\pi/\sqrt K$).\end{thm}

\begin{proof}  Arbitrarily fix two constants $\epsilon>0$ and $K'<K$ with $10\epsilon<K-K'.$

We can choose a point $y\in px$ such that $|py|>\epsilon$ and \begin{equation}\cot_{k_0}(|py|)\ls \cot_{K'}(|py|-\epsilon).\end{equation}
By our definition of   Ricci curvature $\geqslant (n-1)K$ along $\gamma$, there exists a continuous function family $\{g_{\gamma(t)}\}_{0<t<\ell}\in\mathcal F$  such that $$\oint_{\Lambda_{\gamma(t)}}g_{\gamma(t)}\gs K-\epsilon,\qquad \forall t\in(0,\ell).$$

We take a sufficiently small number $\omega>0.$

For any $t_0\in [|py|,|px|],$ there is a neighborhood $I_{t_0}$ coming from Condition $(RC)$ such that $|I_{t_0}|<\omega.$ All of these neighborhoods form an open covering of $[|py|,|px|]$. Let $I_1,\ I_2,\ \cdots, I_N$ be a finite sub-covering of $[|py|,|px|]$. We take $x_a\in I_a\cap I_{a+1}$ for all $1\ls a\ls N-1 $ and set $y=x_0,\ x=x_N.$ We can assume that $|px_a|<|px_{a+1}|$ for all $0\ls a\ls N-1$.

By Condition $(RC)$, we can find a subsequence $\{\delta_{1,j}\}\subset\{\theta_j\}$ and  an isometry $T_1: \Lambda_{x_1}\rightarrow  \Lambda_{x_0}$ such that (1.1) holds. Next, we can  find a further subsequence $\{\delta_{2,j}\}\subset\{\delta_{1,j}\}$ and an isometry $T_2: \Lambda_{x_2}\rightarrow  \Lambda_{x_1}$ such that (1.1) holds. After a finite step of these procedures, we get a subsequence
$\delta=\{\delta_{j}\}\subset \{\delta_{N-1,j}\}\subset\cdots\subset\{\theta_j\}$ and a family isometries  $\{T_{a+1}\}_{a=0}^{N-1}$,  $T_{a+1}:\Lambda_{x_{a+1}}\rightarrow\Lambda_{x_a}$ such that, for each $a=0,1,\cdots,N-1,$ \begin{equation*}\begin{split} |\exp&_{x_{a}}(\delta_j l_{1}T_{a+1}\xi),\ \exp_{x_{a+1}}(\delta_j l_{2}\xi)|\\ \ls  &|x_ax_{a+1}|+\frac{(l_1-l_2)^2}{2\cdot|x_ax_{a+1}|}\cdot\delta_j^2
\\ & -\frac{\big(g_t(\xi)-\epsilon\big)\cdot|x_ax_{a+1}|}{6}\cdot\big(l_1^2+l_1\cdot l_2+l^2_2\big)\cdot\delta^2_j+o(\delta_j^2)\end{split}\end{equation*} for any $l_1,l_2\gs0$ and any $\xi\in \Lambda_{x_{a+1}}$.

\noindent\textbf{Claim:} For all $0\ls a\ls N-1,$ we have $$\oint_{\Lambda_{x_a}} H^\delta_{x_a}f\ls \cot_{K'}(|px_a|-\epsilon),$$  as $\omega$ is sufficiently small.

We will prove the claim by induction argument with respect to $a$.

Firstly, we know  from (3.8) that the case $a=0$ is held.

 Set $q=x_{a},\ r=x_{a+1}$, $\mu=|x_ax_{a+1}|$ and $\ T=T_{a+1}$. Now  we suppose that the claim is held for the case $a$, i.e., $$\oint_{\Lambda_q} H^\delta_qf\ls \cot_{K'}(|pq|-\epsilon).$$
 We need to show the claim is also held for the case $a+1.$

\indent Consider the functions on $\Lambda_r$
\begin{equation}F_l(\xi)=l^2\cdot
H^\delta_qf(T(\xi))+\frac{(l-1)^2}{\mu}-\frac{l^2+l+1}{3}\cdot\mu\cdot\big(g_r(\xi)-\epsilon\big).\end{equation}
>From Lemma 3.2 above, we have \begin{equation}H^\delta_rf\ls
F_l\end{equation} for any $l\gs0.$

On the other hand, from (3.9),
\begin{equation}\begin{split}\oint_{\Lambda_r} F_l&= l^2\cdot\oint_{\Lambda_r} H^\delta_qf\circ T+\frac{(l-1)^2}{\mu}-\frac{l^2+l+1}{3}\cdot\mu\cdot\big(\oint_{\Lambda_r} g_r(\xi)-\epsilon\big)\\
&\ls l^2\cdot\big(\cot_K(|pq|-\epsilon)\big)+\frac{(l-1)^2}{\mu}-\frac{l^2+l+1}{3}\cdot\mu\cdot\bar K
\end{split}\end{equation}
for any $l\gs0,$ where $ \bar K=K-2\epsilon$.

By setting $$C_1=\max_{|py|\ls t\ls |px|}|\cot_{K'}''(t-\epsilon)|,$$ we have
 \begin{equation}\cot_{K'}(|pq|-\epsilon)\ls\cot_{K'}(|pr|-\epsilon)+\mu\big(K'+\cot^2_{K'}(|pq|-\epsilon)\big)+C_1\mu^2.\end{equation} Thus by combining (3.11) and (3.12), we get
\begin{equation}\oint_{\Lambda_r} F_l\ls \cot_{K'}(|pr|-\epsilon)+A_\mu(l),\end{equation} where
\begin{align*}A_\mu(l)=&\mu\big(K'+\cot_{K'}^2(|pq|-\epsilon)\big)+C_1\mu^2+(l^2-1)\cot_{K'}(|pq|-\epsilon)
\\ &+\frac{(l-1)^2}{\mu}-\frac{l^2+l+1}{3}\cdot\mu\cdot \bar K.
\end{align*}
Denote by $B=1/\mu-\mu \bar K/3$ and $\cot=\cot_{K'}(|pq|-\epsilon).$ Note that $$ \cot_{K'}(|px|-\epsilon)\ls\cot\ls \cot_{K'}(|py|-\epsilon).$$
Since $\omega$ is small and $\mu\ls \omega$, we can assume  that $\cot+B>0$.  Choose $\wt l=-(B+\mu \bar K/2)/(\cot+B).$ Then we get \begin{equation*}\begin{split} A_\mu(\wt l)&=\frac{-\big(B+\mu \bar K/2\big)^2+\Big(\mu\big(K'+\cot^2\big)+C_1\mu^2-\cot+B\Big)\cdot\big(\cot+B\big)}{\cot+B}\\
&\ls\frac{K'-\bar K+C_2\mu+C_3\mu^2}{\cot+B},\end{split}\end{equation*} where  $C_2$, $C_3$ are positive constants independent of $\mu,\omega$ (may depending on $\epsilon, K', x$ and $y$).  Using $\mu\ls\omega$, we get $$A_\mu(\wt l)\ls\frac{K'-\bar K+C_2\omega+C_3\omega^2}{\cot+B}\ls0$$
as $\omega$ is sufficiently small. Hence, by combining (3.10), (3.13) and $A_\mu(\wt l)\ls0$, we get $$\oint_{\Lambda_r} H^\delta_rf\ls \oint_{\Lambda_r} F(\wt l)\ls \cot_{K'}(|pr|-\epsilon).$$
This completes the proof of the claim. In particular, we have $$\oint_{\Lambda_x} H^\delta_xf\ls \cot_{K'}(|px|-\epsilon).$$

Thus by the arbitrariness of $\epsilon $ and $K'$ and   a standard diagonal argument, we obtain a subsequence of $\delta$, denoted again by $\delta$, such that
$$\Delta^{\delta}f(x)\leqslant (n-1)\cdot\cot_{K}(|px|).$$ Therefore,  we have  completed the proof of the theorem.
\end{proof}
\section{Rigidity estimates}
We  continue to consider  an $n$-dimensional complete Alexandrov space $M$ without boundary.  Fix a geodesic $\gamma:[0,\ell)\rightarrow M$ with  $\gamma(0)=p$ and denote $f=dist_{p}$.

Let $x\in \gamma\backslash\{p\}$ and $L_x,\ \Lambda_x$  be as above. We  still  assume  that a neighborhood  $U_\gamma$ of $\gamma$ has curvature $\gs k_0$ (for some constant $k_0<0$).

\begin{lem}  Assume  $M$ has Ricci $\geqslant (n-1)K$ along the geodesic $\gamma(t)$. Let $f=dist_p$ and $x$ be an interior point on the geodesic $\gamma(t)$. Given a sequence $\theta=\{\theta_j\}_{j=1}^\infty\in \mathcal S$, if
\begin{equation}\Delta^{\theta'} f(x)= (n-1)\cdot\cot_K(|px|)\end{equation}for any subsequence $\theta'=\{\theta'_j\}$ of $\theta$, then
there exists a subsequence $\delta=\{\delta_j\}$ of $\theta$ such that \begin{equation}H^{ \delta}_xf(\xi)= \cot_K(|px|)\end{equation} almost everywhere  $\xi\in\Lambda_x$.

(If $K>0$, we add assumption $|px|<\pi/\sqrt K$).\end{lem}
 \begin{proof}
At first, we will prove the following claim: \\
\textbf{Claim:} For any $\epsilon>0$, we can find a subsequence
$\{\delta_j\}$ of $\theta$ and an integrable function $h$ on
$\Lambda_x$ such that $$H^\delta_xf\ls h\quad {\rm and }\quad
\oint_{\Lambda_x} \big(h-\cot_K(|px|)\big)^2\ls
\big(3+2|\cot_K(|px|)|\big)\epsilon.$$

  By our definition of   Ricci curvature $\geqslant (n-1)K$ along $\gamma$, there exists a continuous function family $\{g_{\gamma(t)}\}_{0<t<\ell}\in\mathcal F$  such that $$\oint_{\Lambda_{\gamma(t)}}g_{\gamma(t)}\gs K-\epsilon,\qquad \forall t\in(0,\ell).$$
We may assume $g_x\gs k_0$, otherwise, we replace it by $\max\{g_x, k_0\}$.

    By the definition of Condition $(RC)$, we have a neighborhood $I (\subset (0,\ell))$ of $\gamma^{-1}(x)$ such that for arbitrarily taking   a point $w\in \gamma(I)$  with $|pw|<|px|,$
there exists a subsequence  $\wt\delta=\{\wt\delta_j\}$ of $\theta$
and an isometric $T: \Lambda_x\rightarrow\Lambda_w$ such that (1.1)
holds.  By using  Lemma 3.2 and choosing $l=1$, we have
\begin{equation}(g_x- \epsilon)\cdot|xw|\ls H^{\wt\delta}_w
f-H^{\wt\delta}_xf.\end{equation} By (3.2) and the fact that $f$ is
$\cot_{k_0}(|p\cdot|)-$concave, we have $$H^{\wt\delta}_xf\gs
-\cot_{k_0}(|x\gamma(\ell)|)\quad {\rm and}\quad H^{\wt\delta}_wf\ls
\cot_{k_0}(|pw|).$$ Thus by combining these with (4.3) and the fact
$g_x\gs k_0$, we get \begin{equation}|g_x|\ls C_4\end{equation} for
some constant $C_4,$ which may depend on $\epsilon,x,$ and $|I|$.

Choose a point $z\in\gamma(I)$ with $|px|/2<|pz|<|px|$ and $|xz|\ll \min\{\epsilon,|I|\}$. Then, by Condition $(RC)$, there exists a subsequence $\{\delta_j'\}$ of $\theta$  and an isometry  $T:\ \Lambda_x\rightarrow\Lambda_z$ satisfying (1.1). From Theorem 3.3, we can find a subsequence $\{\delta_j\}\subset\{\delta_j'\}$  such that \begin{equation} \Delta^\delta f(z)\ls (n-1)\cdot\cot_K(|pz|).\end{equation}
We set, for any $\xi\in\Lambda_x,$ \begin{align*}\mu&=|xz|,\\ l&=l(\xi)=\big(1/\mu+\frac{\mu}{6}(g_x-\epsilon)\big)\cdot\Big(1/\mu-\mu(g_x-\epsilon)/3+H^\delta_zf(T\xi)\Big)^{-1}\end{align*} and \begin{equation}h_{xz}(\xi)=l^2\cdot H^\delta_zf(T\xi)+\frac{(l-1)^2}{\mu}-\frac{l^2+l+1}{3}\mu(g_x-\epsilon).\end{equation}

 By noting (4.4) and that  $$-\cot_{k_0}(|x\gamma(\ell)|)\ls -\cot_{k_0}(|z\gamma(\ell)|)\ls H^\delta_zf\ls \cot_{k_0}(|pz|)\ls\cot_{k_0}(|px|/2),$$
 we get  $l(\xi)>0$ for $\mu$ is sufficiently small. Thus by Lemma 3.2, we have $$H^\delta_xf\ls h_{xz}.$$ Consequently,
$$H^\delta_xf\ls h,\quad {\rm on}\ \Lambda_x,$$ where
 $h=\min\{h_{xz}, \cot_{k_0}(|px|)\}$. Then, by combining this with (4.1), we get
  \begin{equation} \oint_{\Lambda_x} h\gs \cot_K(|px|).\end{equation}
Therefore, by (4.5) and (4.7), there holds\begin{equation}\begin{split}\oint_{\Lambda_x} h-\oint_{\Lambda_z} H^\delta_zf&\gs\cot_K(|px|)-\cot_K(|pz|)\\ &\gs-\mu\big(K+\cot_K^2(|px|)\big)-C_{5}\mu^2,\end{split}\end{equation}where $$C_5=\max_{|pz|\ls t\ls |px|} |\cot_K''(t)|\ls \max_{|px|/2\ls t\ls |px|} |\cot_K''(t)|.$$

 On the other hand, rewriting the equation (4.6), we have \begin{align*}\Big(1/\mu&-\mu(g_x-\epsilon)/3+H^\delta_{z}f\circ T\Big)\cdot h_{xz}\\ &= -(g_x-\epsilon)+H^\delta_{z}f\circ T\cdot\Big(1/\mu-\mu(g_x-\epsilon)/3\Big) +\big(\mu(g_x-\epsilon)\big)^2/12.\end{align*}
 By the facts that $h\ls h_{xz}$ and $1/\mu-\mu(g_x-\epsilon)/3+H^\delta_{z}f\circ T>0$, we get
  \begin{align*}\Big(1/\mu&-\mu(g_x-\epsilon)/3+H^\delta_{z}f\circ T\Big)\cdot h\\ &\ls -(g_x-\epsilon)+H^\delta_{z}f\circ T\cdot\Big(1/\mu-\mu(g_x-\epsilon)/3\Big) +\big(\mu(g_x-\epsilon)\big)^2/12.\end{align*}
 That is,  \begin{equation}\Big(1/\mu-D\Big)\cdot (h-H^\delta_{z}f\circ T)\ls -(g_x-\epsilon)-h^2 +\big(\mu(g_x-\epsilon)\big)^2/12,\end{equation}
 where $D=\mu(g_x-\epsilon)/3-h$.

Denote that $C_6=\max|D|=\max|h+\mu(g_x-\epsilon)/3|$, which is  independent of $\mu.$  Thus we get
\begin{equation}\begin{split}
\oint_{\Lambda_x}\frac{\epsilon-g_x}{1/u-D}&=\oint_{\Lambda_x}\frac{(\epsilon-g_x)^+}{1/u-D}-\oint_{\Lambda_x}\frac{(\epsilon-g_x)^-}{1/u-D}\\
&\ls \frac{\oint_{\Lambda_x}(\epsilon-g_x)^+}{1/u-C_6}-\frac{\oint_{\Lambda_x}(\epsilon-g_x)^-}{1/u+C_6}\\
&=\frac{1/\mu\oint_{\Lambda_x}(\epsilon-g_x)+C_{6}\oint_{\Lambda_x}|g_x-\epsilon|}{1/\mu^2-C^2_{6}}.
\end{split}\end{equation} By (4.4), (4.10) and the Ricci curvature condition that $\oint_{\Lambda_x}g_x\gs K-\epsilon$, we have
\begin{equation}
\oint_{\Lambda_x}\frac{\epsilon-g_x}{1/u-D}\ls \mu(2\epsilon-K)+C_7\mu^2,\end{equation}
where constant $C_7$ is independent on $\mu$.\\
>From (4.9) and (4.4), we get
\begin{equation}\oint_{\Lambda_x} h-\oint_{\Lambda_x} H^\delta_zf\circ T\ls\mu(2\epsilon-K)+C_7\mu^2-\frac{\oint_{\Lambda_x} h^2}{1/\mu+C_{6}}+\frac{(C_{4}+\epsilon)^2\mu^2}{1/\mu-C_{6}}. \end{equation}By combining  (4.8), (4.12) and noting that $T$ is an isometry,   we have $$\oint_{\Lambda_x} h^2\ls \cot^2_K(|px|)+2\epsilon+C_{8}\mu,$$where constant $C_8$ is independent on $\mu$. Therefore, \begin{equation}\oint_{\Lambda_x} h^2\ls \cot^2_K(|px|)+3\epsilon\end{equation} as $\mu$ suffices small.

Note that  (4.12) implies $$\oint_{\Lambda_x} h\ls \oint_{\Lambda_x} H^\delta_zf\circ T+C_9\mu,$$where constant $C_9$ is independent on $\mu$.
Using (4.5) and  noting that $T$ is an isometry, we have $$\oint_{\Lambda_x} h\ls\cot_K(|pz|)+C_9\mu\ls\cot_K(|px|)+\mu\big(K+\cot^2_K(|pz|)\big)+C_9\mu.$$
Since  $|px|/2<|pz|<|px|$, we have $$\oint_{\Lambda_x} h\ls\cot_K(|px|)+C_{10}\mu,$$ where constant $C_{10}$ is independent on $\mu$.
Thus, when $\mu$ is sufficiently small, we get \begin{equation}\oint_{\Lambda_x} h\ls\cot_K(|px|)+\epsilon.\end{equation}
By combining (4.7) and (4.14), we obtain \begin{equation}\cot_K(|px|)\cdot\oint_{\Lambda_x} h\gs \cot^2_K(|px|)-\epsilon\cdot|\cot_K(|px|)|.\end{equation}
Hence, by (4.13) and (4.15), we have $$\oint_{\Lambda_x} \big(h-\cot_K(|px|)\big)^2\ls \big(3+2|\cot_K(|px|)|\big)\cdot\epsilon.$$
This completes the proof of the claim.

Now let us continue the proof of the lemma.

  Given any $\epsilon_1>0$, the  above claim implies that the measure \begin{align*}\nu&\big(\{\xi\in \Lambda_x:\ H^\delta_xf\gs\cot_K+\epsilon_1\}\big)\\ &\ls\nu\big(\{\xi\in \Lambda_x:\ \big| h-\cot_K(|px|)\big|\gs\epsilon_1\}\big)\ls \big(3+2|\cot_K(|px|)|\big)\epsilon/\epsilon_1^2.\end{align*}
Letting $\epsilon\to 0^+$, by  a standard diagonal argument, we can obtain a subsequence of $\delta$, still denoted by $\delta$, such that  $$\nu\big(\{\xi\in \Lambda_x:\ H^{ \delta}_xf\gs\cot_K+\epsilon_1\}\big)=0.$$
By the arbitrariness of $\epsilon_1$, after a further diagonal argument, we obtain a  subsequence  of $\delta$, denoted by $\delta$ again, such that $$\nu\big(\{\xi\in \Lambda_x:\ H^{ \delta}_xf>\cot_K\}\big)=0.$$ Thus we have  $$H^{ \delta}_xf\ls \cot_K(|px|)$$ almost everywhere in $\Lambda_x$.

 Finally , by combining  (4.1) and the definition of $\Delta^\delta f$,  we conclude that  $$H^{ \delta}_xf= \cot_K(|px|)$$ almost everywhere in $\Lambda_x$. Therefore we have completed the proof of the lemma.
\end{proof}

In order to deal with the zero-measure set in  the above Lemma, we need the following \emph{segment inequality} of Cheeger and Colding [CC1]. See also [R] for a statement that is stronger than the following proposition.
\begin{prop} (Segment inequality)

Let $M$ be an $n$-dimensional Alexandrov space with curvature $\gs
k_0$, ( for some constant $k_0<0$). Let $A_1, A_2\subset M$ be two
open sets, and let $\gamma_{y_1,y_2}$ be a geodesic from $y_1$ to
$y_2$ with arc-parametrization. Assume $W\subset M$ is an open set
with $$\bigcup_{y_1\in A_1,\ y_2\in A_2}\gamma_{y_1,y_2} \subset
W.$$

If $e$ be a non-negative integrable function on $W$, then \begin{equation}
\int_{A_1\times A_2}\int_0^{|y_1y_2|}e(\gamma_{y_1,y_2}(s))ds\ls C(n,k_0,D)\cdot D\cdot\big(vol(A_1)+vol(A_2)\big)\int_We,
\end{equation}
where $D=\sup_{y_1\in A_1,\ y_2\in A_2}|y_1y_2|$ and $$C(n,k_0,D)=\big(\sinh(\sqrt{-k_0}D)/\sinh(\sqrt{-k_0}D/2)\big)^{n-1}.$$
\end{prop}

We now define the \emph{upper Hessian} of $f$, $\overline{Hess}_xf:\ T_x\rightarrow\mathbb{R}\cup\{-\infty\}$ by  \begin{equation}\overline{Hess}_xf(v,v)\overset{def}{=}\limsup_{s\to0}\frac{f\circ\exp_x(s\cdot v)-f(x)-d_xf(v)\cdot s}{s^2/2}\end{equation} for any $v\in T_x$.

Clearly, this definition also works for any semi-concave function on $M$. If $u$ is a $\lambda-$concave function, then its upper Hessian $\overline{Hess}_xu(\xi,\xi)\ls \lambda$ for any $\xi\in\Sigma_x$.

 For a semi-concave function $u$, we denote its regular set  $Reg_u$ by
$$Reg_u:=\big\{x\in M: {\rm there\ exists\ Perelman's\ Hessian \ of}\ u\ {\rm at}\ x \big\}.$$ It was
showed in  [Per2] that $Reg_u$ has full measure for any semi-concave
function $u$. It is clear that $\overline{Hess}_xu=Hess_xu$ for any
$x\in Reg_u$.

\begin{defn}  Let $p\in M$. The \emph{cut locus} of $p$, denoted by $Cut_p$, is defined to be the set all of points $x$ in $M$ such that geodesic $px$, from $p$ to $x$, can not be extended.\end{defn}
It was shown in [OS] that $Cut_p$ has zero (Hausdorff) measure (see
also [Ot]).

  Set $ W_p=M\backslash(\{p\}\cup Cut_p).$ For any two points $x,y\in M$ with $ x\not=y$,   a direction from $x$ to $y$ is denoted by  $\ua_x^y$.

The following two lemmas are concerned with   the rigidity part of Theorem 3.3.
\begin{lem}
Let $M$ be an $n$-dimensional Alexandrov space with Ricci curvature $\gs (n-1)K$ and let $f=dist_p$. Suppose that $B_p(R)\backslash\{p\}\subset W_p$ for some $0<R\ls\pi/\sqrt K$ (if $K\ls 0$, we set $\pi/\sqrt K$ to be $+\infty$). Assume that for each $x\in B_p(R)\backslash\{p\}$, there exists a sequence $\theta:=
 \{\theta_j\}_{j=1}^\infty\in \mathcal S$ such that $$\Delta^{\theta'}f(x)=(n-1)\cdot\cot_K(|px|)$$ for any subsequence $\theta'\subset \theta$.

  Then the function $\varrho_K\circ f$ is $(1-K\cdot\varrho_K\circ f)-$concave in $B_p(R)\backslash\{p\}$.
  \end{lem}
\begin{proof} It suffices to show one variable function  $h_p:=\varrho_K\circ f\circ\gamma(s)$ satisfies that
$$h_p''\ls 1-Kh_p$$ for any geodesic $\gamma(s)\subset B_p(R)\backslash\{p\}.$
 Let  $ \chi(s)$ be an continuous function on  an open interval $(a,b)$. Here and in the sequel we write $\chi''(s)\ls B$ for $ s\in (a,b)$   if $\chi(s+\tau)\ls \chi(s)+A\cdot \tau+ B\cdot \tau^2/2+o(\tau^2)$ for some $A\in \mathbb{R}$.  $\chi''(s)<+\infty$ means that $\chi''(s)\ls B $ for some $B\in\mathbb{R}$. If $\chi_1$ is another continuous function on $(a,b)$, then $\chi''\ls \chi_1$ means $\chi''(s)\ls \chi_1(s)$ for all $s\in (a,b).$

Fix a geodesic $\gamma\subset B_p(R)\backslash\{p\}$. Let $x=\gamma(0)$, $y=\gamma(l).$ Without loss of generality, we can assume that $ \gamma$ is the unique geodesic from $x$ to $y$ and $$|px|+|py|+|xy|<2R.$$

We consider the function $u: W_p\rightarrow \mathbb{R}^+\cup\{0\}$, \begin{equation}u(z)=\sup_{\xi\in \Sigma_z}\Big|\overline{Hess}_zf(\xi,\xi)-\cot_K(|pz|)\cdot\sin^2(|\xi,\uparrow^p_z|)\Big|.\end{equation}

For any point $z\in Reg_f\cap B_p(R)$, $\overline{Hess}_zf$ is a bilinear form on $T_z$ and $\overline{Hess}_zf(\uparrow^p_z,\uparrow^p_z)=0$. Let $$\Lambda_z=\{\xi\in\Sigma_z: \angle(\xi,\uparrow_z^p)=\pi/2\}.$$ By Lemma 4.1,
we have $\overline{Hess}_zf(\xi,\xi)=H^{\delta}_zf=\cot_K(|pz|)$ on $\Lambda_z$ for some subsequence $\delta$ of $\theta$,  and hence $u(z)=0.$

  Since $Reg_f$ has full measure in $B_p(R),$  we conclude that $u\equiv0$ almost everywhere in $B_p(R)$.

Given any positive number $\epsilon>0$ such that $$\epsilon\ll \min\big\{|px|,|py|,|xy|, 2R-(|px|+|py|+|xy|)\big\}.$$ Let $x_1\in B_x(\epsilon)$ and $y_1\in B_y(\epsilon)$, and let $\gamma_{x_1,y_1}(s)$ be a geodesic from $x_1$ to $y_1$. By triangle inequality, it is easy to see $$|px_1|+|x_1y_1|+|py_1|<2R$$ as $\epsilon$ is sufficiently small. Thus $\gamma_{x_1,y_1}\in B_p(R).$

Set $u_{x_1,y_1}(s)=u(\gamma_{x_1,y_1}(s)).$ By applying Proposition
4.2 to $A_1=B_x(\epsilon),$ $ A_2=B_y(\epsilon)$, $W=B_p(R)$ and
function $u$, we know that there exist two points $x_1\in
B_x(\epsilon)$ and $y_1\in B_y(\epsilon)$ such that
$u_{x_1,y_1}(s)=0$ almost everywhere on $(0,|x_1y_1|)$.

Consider a $s_0\in(0,|x_1,y_1|)$ such that $u_{x_1,y_1}(s_0)=0$.  Set $z=\gamma_{x_1,y_1}(s_0)$, $\zeta^+=\gamma^+_{x_1,y_1}(s_0)$ and $\zeta^-=\gamma^-_{x_1,y_1}(s_0)$. Then we have $$\overline{Hess}_{z}f(\zeta^+,\zeta^+)=\overline{Hess}_{z}f(\zeta^-,\zeta^-)=\cot_K(|pz|)\cdot\sin^2(|\zeta^+,\uparrow_z^p|).$$
Therefore, for function $\wt f(s)=f\circ \gamma_{x_1,y_1}(s)$, we get
\begin{equation}\begin{split}\wt f(h+s_0)&\ls \wt f(s_0)+h \wt f^+(s_0)+F(s_0)\cdot h^2/2+o(h^2),\\
\wt f(-h+s_0)&\ls \wt f(s_0)-h \wt f^-(s_0)+F(s_0)\cdot
h^2/2+o(h^2),\end{split}\end{equation}
 for any $h>0$, where
 \begin{equation*}F(s_0)=\cot_K(|pz|)\cdot\sin^2(|\zeta^+,\uparrow_z^p|)=\cot_K(|pz|)\cdot\big(1-\cos^2(|\zeta^+,\uparrow_z^p|)\big).
 \end{equation*}

By the first variation formula of arc-length, we have
$$\wt{f}^+(s)=-\cos(|\zeta^+,\uparrow_z^p|)\quad {\rm and}\quad \wt{f}^-(s)=-\cos(|\zeta^-,\uparrow_z^p|).$$
 Note that $\gamma_{x_1,y_1}\in W_p$, $$|\zeta^+,\uparrow_z^p|+|\zeta^-,\uparrow_z^p|=\pi,$$
 which implies that   $\wt f(s)$ is continuously differential.
 Then by combining this with  (4.19), we have $$\wt f''(s)\ls F(s)=\cot_K\wt f(s)\cdot\big(1-\wt{f'}^2(s)\big)$$ for almost everywhere $s\in (0,|x_1y_1|)$.
Thus the function $\wt h(s)=\varrho_K\circ\wt f(s)$ satisfies $$\wt h''(s)\ls 1-K\wt h(s)$$
for almost everywhere $s\in (0,|x_1y_1|)$.
On the other hand,  the fact $f$ is semi-concave  implies that $\wt h''(s)<+\infty$ for all $s\in (0,|x_1y_1|)$. Thus, from 1.3(3) in [PP], we have $$\wt h''\ls 1-K\wt h.$$

Letting $\epsilon\to0^+$, we can get point sequences $\{x_i\}$ and  $\{y_i\}$ such that $x_i\to x$, $y_i\to y$ and $$\wt  h_i''\ls 1-K\wt h_i,$$ where $\wt h_i=\varrho_K\circ f\circ\gamma_{x_i,y_i}(s)$.  Since the geodesic from $x$ to $y$ is unique, there exists a subsequence of  geodesics $\gamma_{x_i,y_i}$, which converges to geodesic $\gamma$ uniformly.  Hence $\wt h_i$ converges to $h$ uniformly, and the desired result follows  from 1.3(4) in [PP]. Therefore, we have completed the proof.
\end{proof}
\begin{lem}
 Let $\sigma(t)$ and $\varsigma(t)$ be  two geodesics in $B_p(R)$ with $\sigma(0)=\varsigma(0)=p$, and let  $$\varphi(\tau,\tau')=\wa_K \sigma(\tau)p\varsigma(\tau')$$ be the comparison angle of $\angle \sigma(\tau)p\varsigma(\tau')$ in the $K-$plane. Then, under the same assumptions as Lemma 4.4, we have $\varphi(\tau,\tau')$ is non-increasing with respect to $\tau$ and $\tau'$.

(If $K>0$, we add the assumption that
$\tau+\tau'+|\sigma(\tau)\varsigma(\tau')|<2\pi/\sqrt K$).
\end{lem}
\begin{proof}
Firstly, we claim that for any triangle $\triangle pxy$, (if $K>0$, we  assume that $|px|+|py|+|xy|<2\pi/\sqrt K$),  there exists a comparison triangle $\triangle \bar p\bar x\bar y$ in the $K-$plane $M^2_K$ such that
\begin{equation}\angle\bar p\bar x\bar y\ls\angle pxy,\quad \angle\bar p\bar y\bar x\ls\angle pyx.\end{equation}

Indeed,  for any triangle $\triangle pxy\in B_p(R)$, there exists a
triangle $\triangle \widehat p\widehat x\widehat y$ in  $M^2_K$ such
that $$|\widehat p\widehat x|=|px|,\qquad |\widehat x\widehat
y|=|xy|,\qquad \angle\widehat p\widehat x\widehat y=\angle pxy,$$
and by Lemma 4.4, we have $$|\widehat p\widehat y|\gs |py|.$$ So by
an obvious reason, we  get the required triangle $\triangle \bar p
\bar x\bar y$.

 Fix $\tau'>0$ and write $\varsigma=\varsigma(\tau')$. We only need  to show $\varphi(\tau):=\varphi(\tau,\tau')$ is non-increasing with respect to $\tau.$

 Let $\triangle \bar \sigma(\tau)\bar p\bar \varsigma$ be a comparison triangle of $\triangle \sigma(\tau) p \varsigma$ in the $K-$plane $M^2_K$ and extend the geodesic $\bar p\bar \sigma(\tau)$ slightly longer to $\bar \sigma(\tau+s)$ for small $s>0$.

 Since the function $dist_\varsigma$ is $\lambda-$concave for some number $\lambda\in \mathbb{R},$ we have \begin{equation}|\varsigma \sigma(\tau+s)|\ls |\varsigma \sigma(\tau)|+s\cdot\big(-\cos\angle \sigma(\tau+s) \sigma(\tau)\varsigma\big)+s^2\lambda/2.\end{equation}
On the other hand, we have
\begin{equation}|\bar\varsigma \bar\sigma(\tau+s)|= |\bar\varsigma \bar\sigma(\tau)|+s\cdot\big(-\cos\angle \bar\sigma(\tau+s) \bar\sigma(\tau)\bar\varsigma\big)+s^2\bar\lambda/2+o(s^2)\end{equation}for some number $\bar\lambda\in \mathbb{R}.$
Note from (4.20) that
 $$\angle\bar \sigma(\tau+s)\bar \sigma(\tau)\bar \varsigma\gs \angle \sigma(\tau+s) \sigma(\tau)\varsigma. $$
By combining this with (4.21), (4.22) and  $|\varsigma \sigma(\tau)|=|\bar\varsigma \bar\sigma(\tau)|$,  we have   \begin{equation}|\varsigma \sigma(\tau+s)|\ls|\bar \varsigma\bar \sigma(\tau+s)|+(-\lambda+\bar\lambda)s^2+o(s^2).\end{equation}
Now, if $K>0$, by cosine law in $M^2_K$, we have
\begin{align*} \cos\wa_K \sigma(\tau+s)&p\varsigma-\cos\wa_K \sigma(\tau)p\varsigma\\ &=\frac{\cos(\sqrt K|\varsigma \sigma(\tau+s)|)-\cos(\sqrt K|\bar \varsigma\bar \sigma(\tau+s)|)}{\sin(\sqrt K|p\sigma(\tau+s)|)\cdot\sin(\sqrt K|p\varsigma|)}\\
&\gs \frac{-(\lambda+\bar\lambda)}{\sin(\sqrt K|p\sigma(\tau+s)|)\cdot\sin(\sqrt K|p\varsigma|)}\cdot s^2.\end{align*}
Hence, we get $$\frac{d^+}{d\tau}\cos\wa_K\sigma(\tau)p\varsigma\gs0.$$If $K\ls 0$, using a similar argument, we can get $\frac{d^+}{d\tau}\wa_K\sigma(\tau)p\varsigma\ls0.$ Therefore we have completed the proof of  the lemma.
\end{proof}

\section{Maximal diameter theorem}
The main purpose of this section is to prove Theorem 1.8.

Bonnet--Myers' theorem asserts that  if an $n$-dimensional Riemannian manifold  has $Ric\gs n-1$, then its diameter $\ls \pi.$  Furthermore, its fundamental group  is finite.

The first assertion, the diameter estimate, has been extend to metric measure space with $CD(n,n-1)$ (see [S2]) or $MCP(n,n-1)$ (see [O1]). Since our condition $Ric\gs n-1$ implies the curvature-dimension condition $CD(n,n-1)$,  the first assertion of Bonnet--Myers' theorem also holds on  an $n$-dimensional Alexandrov space $M$ with $Ric(M)\gs n-1$ and $\partial M=\varnothing$.

Now we consider the second assertion: finiteness  of the fundamental  group.
\begin{prop}Let $M$ be  an $n$-dimensional Alexandrov space without boundary and  $Ric(M)\gs n-1$.
   The order of  fundamental group of $M$, ${\rm ord} \pi_1(M) ,$ satisfies$${\rm ord} \pi_1(M)\ls \frac{\omega_n}{vol(M)}$$ where $\omega_n$ is the volume of $n$-dimensional standard sphere $\mathbb{S}^n$. In particular, if add assumption $vol(M)>\omega_n/2$, $M$ is simply connected.
\end{prop}
\begin{proof}
Let $\wt M$ be the universal covering of $M$. We have $Ric(\wt M)\gs n-1$. Therefore, by Bishop--Gromov volume comparison theorem (see Corollary A.3 in Appendix),  we get
$${\rm ord} \pi_1(M)\cdot vol(M)=vol(\wt M)\ls \omega_n.$$
This completes the proof.\end{proof}
\indent Now, we  are in position to prove Theorem 1.8. We rewrite it as following
\begin{thm}
 Let $M$ be  an $n$-dimensional Alexandrov space with $Ric(M)\gs n-1$ and $\partial M=\varnothing$. If $diam(M)=\pi$ , then $M$ is isometric to  suspension $[0,\pi]\times_{\sin}N,$ where $N$ is an Alexandrov space with curvature $\gs1.$
\end{thm}

\begin{proof} Takes two points $p,q\in M$ such that $|pq|=\pi.$

Exactly as in Riemannian manifold case, by using Bishop--Gromov volume comparison theorem, we have the following assertions: \\
\textbf{Fact:}\indent (i)  For any point $x\in M$, there holds $|px|+|qx|=\pi.$  This implies $W_p=W_q=M\backslash\{p,q\}.$\\
\indent  (ii) For any $x\in M$,  we can extend the geodesic $px$ to a geodesic from $p$ to $q$. We will denote it by $ pxq.$\\
\indent  (iii) For any non-degenerate triangle $\triangle pxy$, we have $|px|+|py|+|xy|<2\pi.$\\
\indent  (iv) For any direction $\xi\in \Sigma_p$, there exists a geodesic $\gamma_{\xi}$ such that $\gamma_\xi(0)=p, \ \gamma_\xi^+(0)=\xi$ and its length is equal to $\pi.$

 Indeed, the first assertion (i) is an immediate consequence of Bishop--Gromov volume comparison theorem (see, for example, page 271 in [P]). Gluing geodesics $px$ and $qx$, the result curve has length $=\pi=|pq|$. Thus it is a geodesic. This proves the second assertion (ii). The third assertion (iii) follows directly from triangle inequality $$|px|+|py|+|xy|<|px|+|py|+|qx|+|qy|\ls2\pi.$$
To show (iv), we consider a sequence of direction $\xi_i\in\Sigma_p$ such that $\xi_i\to\xi$ and there exists geodesics $\alpha_i$ with $\alpha_i(0)=p$ and $\alpha^+_i(0)=\xi_i$. From (ii), we can extend each $\alpha_i$ to a new geodesic with length $=\pi$, denoted by $\alpha_i$ again.
By Arzela--Ascoli Theorem, we can take a limit from some subsequence of $\alpha_i$. Clearly, the limit is the desired geodesic. This proves the last assertion (iv).

Let $f=dist_p$ and $\bar f=dist_q$. For any point $x\not=p,q$, we
set $\Lambda_x\subset \Sigma_x$ all of directions which are vertical
with the geodesic $pxq$.

Fix a sequence $\theta=\{\theta_j\}_{j=1}^\infty\in\mathcal S$. By Theorem 3.3, we can find a subsequence $\delta\subset\theta$ such that  \begin{equation}\Delta^\delta f(x)\ls (n-1)\cdot\cot(|px|)\quad {\rm and}\quad \Delta^\delta \bar f(x)\ls (n-1)\cdot\cot(|qx|).\end{equation}
 The above fact (i) implies $f+\bar f=\pi$. Thus \begin{equation}H^\delta_x \bar f(\xi)= -\liminf_{s\to0\ s\in\delta}\frac{ f\circ\exp_x(s\cdot\xi)- f(x)}{s^2/2}.\end{equation}
By Definition 3.1, we have  $H^{\delta'}_x f\gs-H^{\delta}_x\bar f$ for any subsequence $\delta'\subset \delta$.
Hence, by combining this with (5.1) and the definition of $\Delta^\delta f$, we get
$$\Delta^{\delta'} f(x)\gs -\Delta^{\delta} \bar f(x)\gs -(n-1)\cdot\cot(|qx|)=(n-1)\cdot\cot(|px|).$$
Note also that  $$\Delta^{\delta'} f(x)\ls \Delta^{\delta} f(x).$$
By combining this with (5.1), this implies that $$\Delta^{\delta'}
f(x)=(n-1)\cdot\cot(|px|)$$ for any subsequence $\delta'\subset \delta$.

>From Lemma 4.4, $-\cos f$ is $\cos f-$concave in
$B_p(\pi)\backslash\{p\}=W_p$. Given any geodesic
$\sigma(s):[0,L]\rightarrow W_p$ with $L<\pi$, we have
\begin{equation} (-\cos f\circ\sigma)''(s)\ls \cos
f\circ\sigma(s),\qquad \forall s\in(0,L).
\end{equation}
Similarly, $-\cos \bar f$ is $\cos \bar f-$concave in $W_q=W_p$ and
\begin{equation}
(-\cos \bar f\circ\sigma)''(s)\ls \cos \bar f\circ\sigma(s),\qquad
\forall s\in(0,L).
\end{equation}
Since $f+\bar f=\pi$, $\cos f=-\cos\bar f$, by combining this with
(5.3) and (5.4), we get \begin{equation} (-\cos f\circ\sigma)''(s)=
\cos f\circ\sigma(s),\qquad \forall s\in(0,L).
\end{equation}

Denote by $$M^+=\big\{x\in M:\ f(x)\ls \pi/2\big\},\qquad  M^-=\big\{x\in M:\ f(x)\gs \pi/2\big\}$$
and $N=M^+\cap M^-=\{x\in M:\ f(x)= \pi/2\}.$ Set $$v_x=({\rm geodesic}\  pxq)\cap N,$$
which is consisting of a single point.

We claim that $N$ is totally geodesic in $M$.

 Indeed, take any two points $v_1,v_2 \in N$ with $|v_1v_2|<\pi$. Let $\sigma(s)$ be a geodesic connected $v_1$ and $v_2$. By (5.5) and noting that $$\cos f(v_1)=\cos f(v_2)=0,$$ we have $\cos f\circ\sigma(s)\equiv0.$ This tells us $\sigma\subset N$ and $N$ is totally geodesic.

  Now we are ready to prove that $M$ is isometric to suspension $[0,\pi]\times_{\sin}N$. Consider any two points $x,y\in M\backslash\{p,q\}$.

  If $x,y\in M^+$,  we know from Lemma 4.5 that  \begin{equation} \wa_1 xpy\gs \wa_1v_xpv_y\qquad {\rm and}\qquad \wa_1 xqy\ls \wa_1v_xqv_y.\end{equation}
Note from Fact (i) that  $$\wa_1 xpy=\wa_1 xqy.$$
Thus we obtain \begin{equation}\wa_1 xpy=\wa_1v_xpv_y.\end{equation}

 Clearly, if $x,y\in M^-$, the same  argument also deduces the equality (5.7).

While  if $x\in M^+$ and   $y\in M^-$, by Lemma 4.5 again, we have
\begin{equation*} \wa_1 xpy\gs \wa_1v_xpy=\wa_1v_xpv_y\qquad {\rm and}\qquad \wa_1 xpy\ls \wa_1xpv_y=\wa_1v_xpv_y,\end{equation*}
which  implies the equality (5.7).

Then by applying the cosine law to the comparison triangle, we get $$\cos(|xy|)=\cos(|px|)\cdot\cos(|py|)+\sin(|px|)\cdot\sin(|py|)\cos\wa_1 v_xpv_y.$$
This proves that $M$ is isometric to  suspension $[0,\pi]\times_{\sin}N.$

It remains to show  that $N$ has curvature $\gs1.$

We define a map $\Phi: N\rightarrow \Sigma_p$ by $$\Phi(v)=\ua_p^v,\qquad \forall v\in N.$$
Since $N\subset W_p$ and $|pv|=\pi/2$ for all $v\in N$, $\Phi$ is well defined.

Given two points $v_1,v_2\in N$, for any $x_1\in M$ lies in geodesic $pv_1q$ and any $x_2\in M$ lies in geodesic $pv_2q$, the equality (5.7) implies $$\wa_1 x_1py_1=\wa_1 v_1pv_2=|v_1 v_2|.$$
Since $\angle v_1pv_2=\lim_{x_1\to p, x_2\to p}\wa_1 x_1py_1$, we have $$|\ua_p^{v_1}\ \ua_p^{v_2}|_{\Sigma_p}=|v_1 v_2|.$$
This shows that $\Phi$ is an  isometrical embedding. On the other hand, by Fact (iv), $\Phi$ is surjective. Therefore, $\Phi$ is an isometry. Thus $N$ has curvature $\gs1$. Therefore, we have completed the proof of the theorem.
\end{proof}

\begin{cor}
 Let $M$ be  an $n$-dimensional Alexandrov space with $Ric(M)\gs n-1$ and $\partial M=\varnothing$. If $rad(M)=\pi$, then $M$ is isometric to the sphere $\mathbb S^n$ with standard metric.
\end{cor}
\begin{proof} For any point $p\in M$, there exists a point $q$ such that $|pq|=\pi$. From the proof of theorem 5.2, we have that $-\cos dist_p$ is $\ \cos dist_p-$concave in $B_p(\pi)\backslash\{p\}$. Thus $M$ has curvature $\gs1$. It is well-known (see,for example, Lemma 10.9.10 in [BBI]) that  an $n$-dimensional Alexandrov space with curvature $\gs1$ and $rad=\pi$ must be isometric to  the sphere $\mathbb S^n$ with standard metric.
\end{proof}
 \begin{rem}Colding in [C] had proved the corollary  for  limit spaces of   Riemannian manifolds. That is, if  $M_i$ is a sequence of $m-$dimensional Riemannian manifolds with $ Ric_{M_i}\ge m-1$ and converging to a metric space $X$ with $rad_X=\pi$, then $X$ is isometric to  the sphere $\mathbb S^{m'}$ with standard metric for some integer $m'\ls m.$\end{rem}

\section{Splitting theorem}
In this section,  $M$ will always denote an $n$-dimensional Alexandrov space with curvature bounded below locally, $Ric(M)\gs0$ and $\partial M=\varnothing.$ The main purpose of this section is to  prove Theorem 1.7.

 A curve $\gamma:[0,+\infty)\rightarrow M$ is called a \emph{ ray} if $|\gamma(s)\gamma(t)|=s-t$ for any $0\ls t<s<+\infty.$
A curve $\gamma:(-\infty,+\infty)\rightarrow M$ is called a \emph{ line} if $|\gamma(s)\gamma(t)|=s-t$ for any $-\infty<t<s<+\infty.$ For a line $\gamma$, obviously, $\gamma|_{[0,+\infty)}$ and $\gamma|_{(-\infty,0]}$ form two rays.

Given a ray $\gamma(t)$, we define the \emph{Busemann function} $b_{\gamma}$ for $\gamma$ on $M$ by $$b_{\gamma}(x)=\lim_{t\to+\infty}\big(t-|x\gamma(t)|\big).$$
Clearly, it is well-defined and is a 1-Lipschitz function.

>From now on, in this section, we fix a line $\gamma(t)$ in $M$ and set $\gamma_+=\gamma|_{[0,+\infty)}$, $\gamma_-=\gamma|_{(-\infty,0]}$. Let $b_+$ and $b_-$ be the Busemann functions for rays $\gamma_+$ and $\gamma_-$, respectively.

Let us recall what is the proof of the splitting theorem in the smooth case. When $M$ is a smooth Riemannian manifold, Cheeger--Gromoll in [CG] used the standard Laplacian comparison and the maximum principle to conclude that $b_+$ and $b_-$ are harmonic on $M$.  Then the elliptic regularity theory implies that they are smooth. The important step is to use Bochner formula to show that both $\nabla b_+$ and $\nabla b_-$ are parallel. Consequently, the splitting theorem follows directly from de Rham decomposition theorem. In [EH], Eschenburg--Heintze gave a proof avoiding the elliptic regularity; while the Bochner formula is essentially used. But for the general Alexandrov spaces case, the main difficulty is the lack of Bochner formula.

We begin with a lemma which was proved by Kuwae and Shioya for Alexandrov spaces with $MCP(n,0)$ and hence for Alexandrov spaces with nonnegative Ricci curvature.
(See lemma 6.5 and the proof of theorem 1.3 in [KS1]).\begin{lem}$b_+(x)+b_-(x)\equiv0$, on $M$. \end{lem}

\begin{lem}For any point  $x\in M$, there exists a unique line $\gamma_x$ such that $x=\gamma_x(0)$ and $b_+\circ\gamma_x$ is a linear function with $(b_+\circ\gamma_x)'=1$.\end{lem}
\begin{proof}\emph{Existence.} If $x\in \gamma$, then we can write $x=\gamma(t_0)$. Hence we set $\gamma_x(t)=\gamma(t+t_0),$ which is a desired line.

We then consider the case  $x\not\in\gamma$. Let $\sigma_{t,+}(s)$
be a geodesic from $x$ to $\gamma_+(t)$. By using Arzela--Ascoli
Theorem, we can take   a sequence  $t_j\to+\infty$ such that
$\sigma_{t_j,+}$ converges to  a limit curve $
\sigma_{\infty,+}(s):[0,+\infty)\rightarrow M$. It is easy to check
( see, for example,  page 286 in [P]) that $\sigma_{\infty,+}$ is
1-Lipschitz and
$$b_+\circ\sigma_{\infty,+}(s)=s+b_+\circ\sigma_{\infty,+}(0)=s+b_+(x),\quad
{\rm for\ all} \ s\gs0.$$ By a similar construction, we can obtain a
1-Lipschitz curve $\sigma_{\infty,-}(s'):(-\infty,0]\rightarrow M$
such that $\sigma_{\infty,-}(0)=x$ and
$$b_-\circ\sigma_{\infty,-}(s')=-s'+b_-(x), \quad {\rm for\ all} \ s'\ls0.$$ Let $\sigma_\infty=\sigma_{\infty,+}\cup \sigma_{\infty,-}: (-\infty,+\infty)\rightarrow M$. This is a 1-Lipschitz curve. By Lemma 6.1, we have\begin{equation}b_+\circ\sigma_\infty(s)=s+b_+(x),\quad {\rm for\ all} \ s\in(-\infty,+\infty).\end{equation}

 Then for any $-\infty<t<s<\infty$, by (6.1), we get $$s-t=b_+\circ\sigma_\infty(s)-b_+\circ\sigma_\infty(t)\ls |\sigma_\infty(s)\ \sigma_\infty(t)|\ls s-t.$$
Thus  $\sigma_\infty$ is a line. The equation (6.1) shows that it is a desired line.\\

\emph{Uniqueness.} Argue by contradiction. Suppose that there exist two such lines $\gamma_1,\ \gamma_2$.

The equations $(b_+\circ\gamma_1)'=(b_+\circ\gamma_2)'=1$ implies
$$b_+\circ\gamma_1(-1)=b_+(x)-1\quad {\rm and}\quad b_+\circ\gamma_2(1)=b_+(x)+1$$
Hence $$b_+\circ\gamma_2(1)-b_+\circ\gamma_1(-1)=2.$$ Since $b_+$ is
1-Lipschitz, we get
\begin{equation}|\gamma_1(-1)\ \gamma_2(1)|\gs b_+\circ\gamma_2(1)-b_+\circ\gamma_1(-1)=2.\end{equation}
On the other hand,  $$Length\big(\gamma_1([-1,0])\cup\gamma_2([0,1])\big)=2.$$ Thus $\gamma_1([-1,0])\cup\gamma_2([0,1])$ is a geodesic. This contradicts to that $M$ is non-branching.
The proof of the lemma is completed.
\end{proof}

For any point $x\in M$, we take the line $\gamma_x$ in Lemma 6.2. Let \begin{equation*}\begin{split}L_{x}&=\{\xi\in T_{x}\ |\ \angle(\xi,\gamma_x^+(0))=\angle(\xi,\gamma_x^-(0))=\pi/2\},\\
\Lambda_{x}&=\{\xi\in \Sigma_{x}\ |\ \angle(\xi,\gamma_x^+(0))=\angle(\xi,\gamma_x^-(0))=\pi/2\}. \end{split}\end{equation*}
Given a sequence $\theta:=\{\theta_j\}\in\mathcal S$, we define a function $H^\theta_xb_+:\Lambda_x\rightarrow\mathbb{R}$ by
$$H^\theta_xb_+(\xi)\overset{def}{=}\limsup_{s\to0, \ s\in \theta}\frac{b_+\circ\exp_x(s\cdot\xi)-b_+(x)}{s^2/2};$$
and  $$\Delta^\theta b_+(x)\overset{def}{=}(n-1)\cdot\oint_{\Lambda_x} H^\theta_xb_+(\xi).$$

 In the following Lemma 6.3, we will prove that both $b_+$ and $b_-$ are semi-concave. Thus, by lemma 6.1, $H^\theta_xb_+$ is well defined and is locally bounded. It is easy to see that $H^\theta_xb_+$ is measurable, so $\Delta^\theta b_+(x)$ is also well defined.

\begin{lem}$b_+(x)$ is a semi-concave function in $M$. Moreover, for any point $x\in M$ and any  sequence $\theta=\{\theta_j\}\in\mathcal S$, there exists a subsequence $\delta\subset \theta$ such that $\Delta^\delta b_+(x)\ls0.$\end{lem}
\begin{proof} Fix a point $x\in M$, we will construct a semi-concave support function for $b_+$ near $x$.

We take the line $\gamma_{x}$ in Lemma 6.2 and choose a point $p\in\gamma_x$ such that $b_+(p)\ll b_+(x).$

The equation $(b_+\circ \gamma_x)'=1$ implies \begin{equation} b_+(x)-b_+(p)=|px|.\end{equation}
On the other hand, since $b_+$ is 1-Lipschitz, we have \begin{equation}b_+(y)-b_+(p)\ls |py|\end{equation}for any $y\in M$.
By combining (6.3) and (6.4), we know that function $dist_p(\cdot)+b_+(p)$ supports $b_+$ near $x$.

This tells us $b_+$ is a semi-concave function. Furthermore, from Theorem 3.3, we can find a subsequence $\wt \delta\subset\theta$ such that $\Delta^{\wt \delta}b_+(x)\ls (n-1)/|px|.$ By letting $|px|\to\infty$ and a diagonal argument,  we can choose a subsequence $\delta\subset\wt \delta$ such that $\Delta^\delta b_+(x)\ls0.$ Therefore the proof of the lemma is completed.
\end{proof}

The following lemma is similar to Lemma 4.4.
\begin{lem}
Assume that for each point $x\in M$, there exists a sequence $\theta:=
 \{\theta_j\}\in \mathcal S$ such that $\Delta^{\theta'}b_+(x)=0$ for any subsequence $\theta'\subset \theta$.
   Then $b_+$ is a concave function  in $M$.
\end{lem}
\begin{proof} It suffices to show that $b_+$ is concave on an arbitrarily given bounded open set $\Omega\subset M$. Clearly, we may assume $M$ has curvature $\gs k_\Omega$ in $\Omega$ for some constant $k_\Omega$.

In following, we divide the proof into three steps.

\underline{Step 1.}\indent Let $\gamma_x$ be the line in Lemma 6.2. Replacing equation (3.6) and (3.7) by the facts that $|b_+(y)-b_{+}(z)|=|yz|$ for any $y,z\in\gamma_x$ and $b_+$ is 1-Lipschitz,  the same proof in Lemma 3.2 shows that the lemma also  holds when we replace $f=dist_p$ by $b_+$.

\underline{Step 2.}\indent Similar as Lemma 4.1, we want to show
$H_x^{\delta}b_+=0$ almost everywhere in $\Lambda_x$, for some
subsequence $\delta=\{\delta_j\}\subset \theta$.

 We now follow the proof of Lemma 4.1. Firstly, from  Lemma 6.3, we know that both $b_+$ and $b_-$ are semi-concave.  In turn, Lemma 6.1 gives a bound for $H^\theta_x b_+$ .
  Secondly, we use Lemma 3.2 for $b_+$ (i.e., the above Step 1) and replace Theorem 3.3 by the above Lemma 6.3 in the proof of Lemma 4.1.  We repeat the same proof of Lemma 4.1 to get $H_x^{\delta}b_+=0$ almost everywhere in $\Lambda_x$, for some subsequence $\delta\subset\theta$.

 \underline{Step 3.}\indent Following the proof of Lemma 4.4,  we then deduce that  $b_+(x)$ is concave in $\Omega.$ Therefore $b_+(x)$ is concave in $M$ and the proof of the lemma is completed.
 \end{proof}

Now, we are in a position to prove Theorem 1.7.
\begin{proof}[Proof  of  Theorem 1.7]
Given a sequence $\theta=\{\theta_j\}\in\mathcal S,$ from Lemma 6.3, we can find a subsequence $\delta\subset\theta$ such that \begin{equation}\Delta^\delta b_+(x)\ls 0\qquad {\rm and }\qquad \Delta^\delta b_-(x)\ls 0.\end{equation}
By the definition of $\Delta^\delta b_+(x)$ and $\Delta^\delta b_-(x)$, we have $$\Delta^{\delta'} b_+(x)\ls \Delta^\delta b_+(x)\quad {\rm and}\quad \Delta^{\delta'} b_-(x)\ls \Delta^\delta b_-(x)$$for any subsequence $\delta'\subset \delta.$ So (6.5) holds for any subsequence $\delta'\subset\delta$.

On the other hand, by Lemma 6.1 and the definition of $\Delta^\theta
b_+(x)$, we have $$\Delta^\vartheta b_+(x)+\Delta^\vartheta
b_-(x)\gs0$$ for any sequence $\vartheta=\{\vartheta_j\}\in\mathcal
S.$ Therefore, by combining with (6.5), we get
$$\Delta^{\delta'} b_+(x)= 0\qquad {\rm and }\qquad \Delta^{\delta'}
b_-(x)= 0$$for any subsequence $\delta'\subset\delta.$

Then we can apply Lemma 6.4 to conclude that both $b_+$ and $b_-$ are concave. By using Lemma 6.1 again, we deduce that  $b_+\circ\varsigma(s)$ is a linear function on any geodesic $\varsigma(s)$ in $M$. In particular, the level surfaces  $\mathcal L(a):=b_+^{-1}(a)$ are totally geodesic for all $a\in \mathbb{R}$.

Set $N=\mathcal L(0)=b_+^{-1}(0)$. It is an Alexandrov space with curvature bounded below locally.

When $M$ is an Alexandrov space with curvature $\gs -\kappa^2$ for some $\kappa>0$. Mashiko, in [Ma], proved that if there exists a function $u$ such that $u\circ\gamma$ is a linear function for any geodesic $\gamma\subset M$ and  $u\in D^{2,2}$ (see [Ma] for the definition of the class of $D^{2,2}$), then  $M$ is isometric to a direct product $\mathbb{R}\times Y$ over an Alexandrov space $Y$ has curvature $\gs-\kappa^2.$  Later in [AB], Alexander and Bishop  removed the condition $u\in D^{2,2}$.

Since we do not assume that $M$ has a \emph{uniform} lower curvature bound, we adapt Mashiko's  argument  as follows.

 For any $x\in N$ and any $a\in \mathbb{R}$, let $\gamma_x$ be the line obtained in Lemma 6.2.

 Note that $(b_+\circ\gamma_x)(s)'=1$ which implies $\nabla b_+(\gamma_x(s))=\gamma_x^+(s)$. Thus $\gamma_x$ is a gradient curve of $b_+$.

  It is easy to check that $\gamma_x\cap\mathcal L(a)$ is a set of single point.
 We define $\Phi_a: N\rightarrow\mathcal L(a)$ by $\Phi_a(x)=\gamma_x\cap \mathcal L(a).$ $\Phi_a$ and $\Phi_a^{-1}$ are the gradient flows of $b_+$ and $b_-$, respectively.
Since a gradient flow of a concave function is non-expanding, we have that $\Phi_a$ is an isometry.

Now we are ready to show that $M$ is isometric to the direct product $\mathbb{R}\times N$. Consider any  two points $x,y\in M$.

Without loss of generality, we may assume that $x\in N$ and $y\in\mathcal L(a)$ with $a>0$. Let $z=\gamma_y\cap N$, where $\gamma_y$ comes from Lemma 6.2.

We take a $C^1$ curve $\sigma(s)\subset N$ with $\sigma(0)=x$ and $\sigma(Length(\sigma))=z,$ $|\sigma'(s)|=1.$ Define  a new curve $\bar\sigma(s)$ by $$\bar\sigma(s)=\gamma_{\sigma(s)}\Big(\frac{a}{length(\sigma)}\cdot s\Big).$$
Clearly, we have $\bar\sigma(0)=x$, $\bar\sigma(length(\sigma))=\gamma_z(a)=y$ and \begin{equation}b_+(\bar\sigma(s))=\frac{a}{length(\sigma)}\cdot s.\end{equation}

Fixed any $s\in (0,Length(\sigma))$, we set $u=\sigma(s)$ and $v=\bar\sigma(s)$.

We claim that \begin{equation}\angle(\nabla_u b_+,\sigma^+(s))=\angle(\uparrow_u^v,\sigma^+(s))=\pi/2.\end{equation}

Indeed, $$|v\sigma(s')|\gs b_+(v)-b_+(\sigma(s'))=b_+(v)=|vu|$$for any $s'\in (0,Length(\sigma))$. Then by the first variation formula of arc-length, we have \begin{equation} \angle(\uparrow_u^v,\sigma^+(s))\gs\pi/2\quad {\rm and}\quad  \angle(\uparrow_u^v,\sigma^-(s))\gs\pi/2.\end{equation} On the other hand,  \begin{equation} \angle(\uparrow_u^v,\sigma^+(s))+\angle(\uparrow_u^v,\sigma^-(s))=\pi. \end{equation}
Thus the desired (6.7) follows from (6.8) and (6.9).

Now let us calculate the length of the curve $\bar\sigma.$

Clearly, we may assume that a neighborhood of $\bar\sigma$ has curvature $\gs k$  (for some $k<0$).

Fixed $s\in (0,length(\sigma))$. Let $h>0$ be a small number. We set $\bar w=\bar\sigma(s+h)$ and $w=\gamma_{\sigma(s+h)}\big(\frac{a}{length(\sigma)}\cdot s\big)$ (see figure 1).\\
\begin{figure}[h]
\centering
\includegraphics[scale=0.80, bb=115 100 315 200]{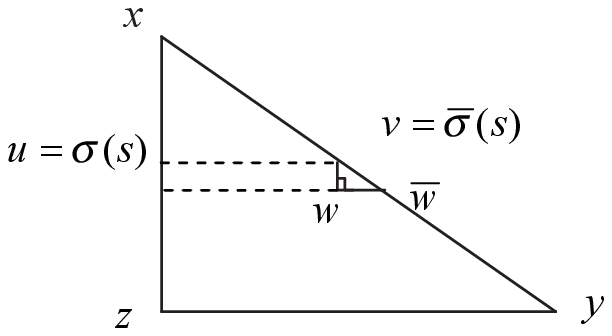}
\caption{ }
\end{figure}\\
 By cosine law  in $0-$plane $\mathbb{R}^2$, we have
 \begin{equation}|\bar \sigma(s+h)\bar\sigma(s)|^2=|v\bar w|^2=|vw|^2+|w\bar w|^2-2|w\bar w|\cdot|vw|\cdot\cos\wa_0 vw\bar w.\end{equation}
Note that
\begin{equation}|v\bar w|=|\sigma(s)\sigma(s+h)|=|\sigma^+(s)\cdot h+o(h)|=h+o(h),\end{equation}
\begin{equation}|w\bar w|=(b_+(\bar w)-b_+(w))=\frac{a}{Length(\sigma)}\cdot h.\end{equation}
By using Lemma 11.2 in [BGP], we have \begin{equation}\wa_k vw\bar w\to  \angle vw\bar w=\pi/2\end{equation} as $h\to0.$ On the other hand, note that \begin{equation}\wa_0 vw\bar w-\wa_k vw\bar w\to 0\end{equation} as $h\to0.$ We have $\cos\wa_0 vw\bar w\to0$ as $h\to0.$\\
Combining this and (6.10)--(6.12), we have
\begin{equation}|\bar \sigma(s+h)\bar\sigma(s)|^2=\Big(1+\big(\frac{a}{Length(\sigma)}\big)^2\Big)\cdot h^2+o(h^2).\end{equation}
Hence, $$|\bar \sigma(s)|^+=\Big(1+\big(\frac{a}{Length(\sigma)}\big)^2\Big)^{1/2}.$$
Similarly, we can get $$|\bar \sigma(s)|^-=\Big(1+\big(\frac{a}{Length(\sigma)}\big)^2\Big)^{1/2}.$$
So \begin{equation}Length(\bar\sigma)=\int_0^{length(\sigma)}|\bar\sigma|'ds=\Big(a^2+\big(Length(\sigma)\big)^2\Big)^{1/2}.\end{equation}

If we take $\sigma_1$ to be a  geodesic $xz$,  we get, from (6.16), that \begin{equation}|xy|^2\ls (Length(\bar\sigma_1))^2=|xz|^2+a^2=|xz|^2+|yz|^2.\end{equation}
While if we take  $\sigma_2$ to be the projection of a geodesic $xy$ to $N$,  we get, from (6.16), that
\begin{equation}|xy|^2=(Length(\sigma_2))^2+a^2\gs|xz|^2+|yz|^2.\end{equation}
The combination of  (6.17) and (6.18) implies that \begin{equation}|xy|^2=|xz|^2+|yz|^2.\end{equation}
This says that $M$ is isometric to the direct product $N\times \mathbb{R}$.

Lastly, we need  prove that $N$ has nonnegative Ricci curvature.

Let $\gamma(t):(-\ell,\ell)\to N$ be a geodesic in $N$. Assume that $N$ has curvature $\gs K$ in a neighborhood of $\gamma$ and for some $K<0$. Otherwise, there is nothing to prove. Hence $M$ has curvature $\gs K$ in a neighborhood of $\gamma$ in $M$.

Let $p$ and $q$ be two interior points in $\gamma$. We denote the tangent spaces, exponential map in $N$ (or $M$, resp.) by $T_pN$, $\exp_p^N$ (or $T_pM=T_{(p,0)}M$, $\exp_p^M=\exp_{(p,0)}^M$, resp.) and $$\Lambda_p^N=\{\xi\in T_p^N:\ \ip{\xi}{\gamma'}=0\}.$$

Let $\Lambda_p^M:=\Lambda_{(p,0)}^M=\{\xi\in T_p^M:\ \ip{\xi}{\gamma'}=0\}$. Then $\Lambda_p^M=S(\Lambda_p^N)$ with vertex $\zeta^\pm$, where $\zeta^\pm$ are the directions along factor $\R$ in $M=N\times\R $. For any $\xi\in T_pN$, we have \be{equation}{
\exp_p^M(\xi,t)=(\exp_p^N(\xi),t).}

Suppose that a family of continuous functions $\{g_{(\gamma(t),0)}(\xi,\eta)\}_{-\ell<t<\ell}$  on $\Lambda_p^M=S(\Lambda_p^N)$ satisfies Condition $(RC)$ on geodesic $(\gamma(t),0)$ in $M$ and $$\int_{\Lambda_p^M} g_{(\gamma(t),0)}(\xi,\eta)d{\rm vol}_{\Lambda_p^M}\gs -\epsilon$$ for a given small number $\epsilon>0$.

Given a sequence $\{\wt s_j\}\in\mathcal S$, the isometry $T: L^M_{(p,0)}\to L^M_{(q,0)}$ and subsequence $\{s_j\}\subset\{\wt s_j\}$ come from the definition of Condition $(RC)$. Recall Petrunin's construction for $T$, we can assume that $T(\zeta^+)=\zeta^+,$ hence $T:L^N_p\subset T_p^N\to L^N_q\subset T_q^N.$

Given a quasi-geodesic $\sigma(s)$ in $N$, setting $\bar\sigma(s)=(\sigma(as),bs)$ for any two number $a,b\in \R$ with $a^2+b^2=1$, we will prove that $\bar\sigma(s)$ is a quasi-geodesic  in $M$.

Let $u(z,r)$ be a $\lambda-$concave function, defined in a neighborhood of $\gamma$ in $M=N\times\R$. So function $u(\cdot,r)$ is $\lambda-$concave
in $N$ and $u(z,\cdot)$ is $\lambda-$concave in $\R$ for all $r\in\R $ and $z\in N$.
 Since $\sigma$ is quasi-geodesic in $N$, we have  $$ u''\big(\sigma(as),r\big)\ls a^2\cdot \lambda$$for all $r\in\R.$ Now $$u''\big(\sigma(as),bs\big)\ls ( a^2+b^2)\cdot \lambda=\lambda.$$ By definition of quasi-geodesic [Pet3], we  get that $\bar\sigma(s)$ is a quasi-geodesic  in $M$.

Fix any nonnegative number $l_1$ and $l_2$. Let $\xi\in \Lambda^N_p.$ For any constant $A\in \R$, we have (see figure 2)\begin{figure}[h]
\centering
\includegraphics[scale=0.80, bb=10 55 370 175]{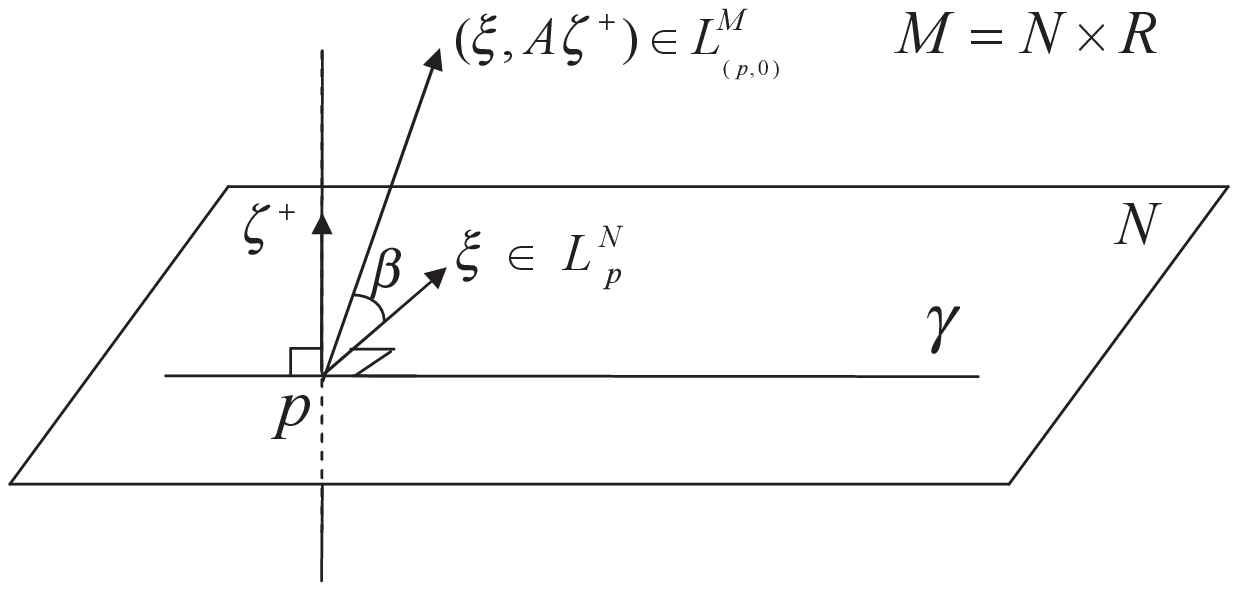}
\caption{ }
\end{figure}
\begin{equation}
\begin{split}
|&\exp^N_p(s_j\cdot l_1\xi),\exp_q^N(s_j\cdot l_2T\xi)|^2\\
&=\big|\big(\exp_p^N(s_j\cdot l_1\xi),s_jl_1A\zeta^+\big),\big(\exp_q^N(s_j\cdot l_2T\xi),s_jl_2A\zeta^+\big)\big|^2-A^2(l_1-l_2)^2\cdot s^2_j\\
&=\big|\exp_p^M\big(s_j\cdot l_1(\xi,A\zeta^+)\big),\big(\exp_q^M\big(s_j\cdot l_2(T\xi,A\zeta^+)\big)\big|^2-A^2(l_1-l_2)^2\cdot s^2_j\\
&\ls |pq|^2+s^2_j\cdot\Big((l_1-l_2)^2-\frac{g_p(\xi,A\zeta^+)\cdot(1+A^2)}{3}|pq|^2(l^2_1+l_1\cdot l_2+l^2_2)\Big)\\&\quad+o(s^2_j).\end{split}
\end{equation}
We set $\beta=\angle\big((\xi,0),(\xi,A\zeta^+)\big)$ and then $A=\tan \beta$, $\beta\in(-\pi/2,\pi/2)$.

For each $t\in(-\ell,\ell)$ and $A\in \R$, we define a function $g_{A,\gamma(t)}:\Lambda^N_{\gamma(t)}\to\R$ by \begin{equation}\be{split}{g_{A,\gamma(t)}(\xi):&= g_{(\gamma(t),0)}(\xi,A\zeta^+)\cdot(1+A^2)\\
&=g_{(\gamma(t),0)}(\xi,A\zeta^+)/\cos^2\beta.}\end{equation}
>From (6.21), for any $A\in\R$, the family of continuous functions $\{g_{A,\gamma(t)}(\xi)\}_{-\ell<t<\ell}$ satisfies Condition $(RC)$ on $\gamma$.

On the other hand, we have \begin{equation}
\begin{split}
-\epsilon&\ls\int_{\Lambda_p^M} g_{(\gamma(t),0)}(\xi,\eta)d{\rm vol}_{\Lambda_p^M}\\
&=\int_{\Lambda_p^N}\int^{\pi/2}_{-\pi/2}g_{(\gamma(t),0)}(\xi,\eta)\cos^{n-2}\beta d\beta d{\rm vol}_{\Lambda_p^N}\\
&= \int_{\Lambda_p^N}\int^{\pi/2}_{-\pi/2}g_{A,\gamma(t)}(\xi)\cos^{n}\beta d\beta d{\rm vol}_{\Lambda_p^N}\\
&= \int^{\pi/2}_{-\pi/2}\int_{\Lambda_p^N}g_{A,\gamma(t)}(\xi)\cos^{n}\beta  d{\rm vol}_{\Lambda_p^N}d\beta.
\end{split}\end{equation}
Thus, we can choose some $A\in \R$ such that  $$\int_{\Lambda_p^N}\big(g_{A,\gamma(t)}(\xi)\big) d{\rm vol}_{\Lambda_p^N}\gs-c_n\cdot\epsilon,$$for some constant $c_n$.
This completes the proof that $N$ has nonnegative Ricci curvature.
Therefore the proof of Theorem 1.7 is completed. \end{proof}

\appendix
\section{}
In the Appendix, we will recall the definition of
\emph{curvature-dimension condition} $CD(n,k)$ which is given by
Sturm [S2] and Lott--Villani [LV1] (see also book [V]). After that we
will present  a proof, due to Petrunin [Pet2], for the statement
that  an $n$-dimensional Alexandrov space with Ricci curvature $\gs
(n-1)K$ and with $\partial M=\varnothing$  must satisfy
$CD(n,(n-1)K).$

 Let $(X,d,m)$ be a metric measure space, where $(X,d)$ is a complete separable metric space.

 Given two measures $\mu$ and $\nu$ on $X$,  a measure $q$
on $X\times X$ is called  a \emph{coupling} (or \emph{transference plan}) of $\mu$ and $\nu$ if $$q(A\times X)=\mu(A)\qquad{\rm and}\qquad q(X\times A)=\nu(A)$$for all measurable $A\subset X$.

 The $L^2-Wasserstein\ distance$ between two  measures $\mu,\nu$ is defined by $$ d^2_W(\mu,\nu)=\inf_q\int_{X\times X}d^2(x,y)dq(x,y)$$
 where infimum runs over all coupling $q$ of $\mu$ and $\nu$. (If $\mu(X)\not=\nu(X)$, we set $d_W(\mu,\nu)=+\infty$.)

  Let $\mathcal P_2(X)$ be the space of all probability measures $\nu$ on $X$ with finite second moments: $$\int_Xd^2(o,x)d\nu(x)<\infty$$ for some (hence all) point $o\in X$.

 $L^2-$Wasserstein space is a complete metric space $(\mathcal P_2(X),d_W)$.
 (see [S1] for the geometry of  $L^2-$Wasserstein space.) Fix a Borel measure $m$ on $X$.
 We denote $L^2-$Wasserstein space by $\mathcal P_2(X,d)$ and its subspace of $m-$absolutely continuous measures is denoted by $\mathcal P_2(X,d,m)$.\\

  Given  $k\in\mathbb{R}$, $n\in(1,\infty]$, $t\in[0,1]$ and two points $x,y\in X$, we define $\beta_t^{(k,n)} $ as follows:\\
   \indent (1)\indent If $0<t\ls 1$, then \begin{equation*}\beta_{t}^{(k,n)}(x,y):=
 \begin{cases}\exp\big(\frac{k}{6}(1-t^2)\cdot d^2(x_0,x_1)\big)& {\rm if }\ n=\infty,\\ \infty & {\rm if }\ n<\infty,\ k>0\ \ {\rm and}\  \alpha\gs \pi,\\  \big(\frac{\sin(t\alpha)}{t\sin\alpha}\big)^{n-1}& {\rm if }\ n<\infty,\ k>0\ \ {\rm and}\ \alpha\in[0, \pi),\\
1& {\rm if }\ n<\infty, \ k=0, \\
 \Big(\frac{\sinh(t\alpha)}{t\sinh\alpha}\Big)^{n-1}& {\rm if }\ n<\infty,\ k<0,\end{cases}
 \end{equation*}  where $\alpha=d(x,y)\cdot\sqrt{|k|/(n-1)}$.\\
\indent (2)\indent $\beta_0^{(k,n)}(x,y)=1.$

The curvature-dimension condition $CD(n,k)$ is defined as follows
(see 29.8 and 30.32 in [V]):
\begin{defn} Let $(X,d,m)$ be a non-branching locally compact complete separable geodesic space equipped with a locally finite measure $m$\footnote{Lott--Villani and Sturm defined curvature dimension condition on general metric measure spaces.}.

  Given two  real numbers $k$ and $n$ with $n>1$,  The metric measure space $(X,d,m)$ is said to satisfy the \emph{curvature-dimension condition }$CD(n,k)$ if and only if
for
each pair  compactly supported  $\mu_0,\mu_1\in\mathcal P_2(X,d,m)$ there exist an optimal coupling $q$ of $\mu_0=\varrho_0 m$ and $\mu_1=\varrho_1 m$,
and a geodesic path\footnote{constant-speed shortest curve defined on $[0,1]$.} $\mu_t: [0, 1]\rightarrow \mathcal P_2(X,d)$ connecting $\mu_0$ and $\mu_1$, with
\begin{equation}\begin{split}
H_{n}(\mu_t|m)\ls& -(1-t)\int_{X\times
X}\Big(\frac{\varrho_0(x)}{\beta^{(k,n)}_{1-t}(x,y)}\Big)^{-1/n}dq(x,y)\\
&\ -t\int_{X\times
X}\Big(\frac{\varrho_1(y)}{\beta^{(k,n)}_{t}(x,y)}\Big)^{-1/n}dq(x,y)
\end{split}\end{equation}for all $t\in[0,1]$, where $H_{n}(\cdot|m): \mathcal P_2(X,d)\rightarrow\mathbb{R}$ is
\emph{R\'{e}nyi entropy functional} with respect to $m$,
$$H_{n}(\mu|m):=-\int_X\varrho^{-1/n}d\mu$$ and $\varrho$ denotes the
density of the absolutely continuous part  in the Lebesgue
decomposition $\mu=\varrho m+\mu^c$ of $\mu.$
\end{defn}

>From now on, in the appendix, $M$ will always denote an
$n$-dimensional Alexandrov space with $Ric(M)\gs (n-1)K$ and
$\partial M=\varnothing.$

Our purpose of this appendix is to prove the following proposition,
which is essentially due to Petrunin [Pet2].
\begin{prop} Let $M$ be  an $n$-dimensional Alexandrov space without boundary and $Ric(M)\gs (n-1)K$. Let $vol$ denote the $n$-dimensional Hausdorff measure on $M$. Then the metric measure space $(M, |\cdot\cdot|, vol)$ satisfies $CD(n,(n-1)K).$\end{prop}

>From [S2], we know that the curvature-dimension condition
$CD(n,(n-1)K)$ implies Bishop--Gromov volume comparison theorem.
Consequently, we get the following
\begin{cor}
Let $M$ be as in above proposition. Then the function, for any $p\in M$,
$$\frac{vol\ B_p(r)}{vol \ B^n_K(r)}$$
is non-increasing in $r>0$, where $B^n_K(r)$ is a geodesic ball of radius $r$ in the $n$-dimensional simply connected Riemannian manifold with constant sectional curvature $K$.
\end{cor}

Before beginning the proof of Proposition A.2, let us review some indispensable materials.

For a continuous function $f$, we define its \emph{Hamilton--Jacobi shift} $\mathcal{H}_tf$ for time $t>0$ by $$\mathcal{H}_tf\overset{def}{=}\inf_{y\in M}\big\{f(y)+\frac{1}{2t}|xy|^2\big\}.$$
Denote by $f_t = \mathcal H_tf$. A solution of $\alpha^+(t)=\nabla_{\alpha(t)}f_t$ is called a $f_t-$\emph{gradient curve}.

Refer to [Pet2] for the existence and uniqueness of $f_t-$gradient curve and basic propositions of Hamilton--Jacobi shifts. Now we list only  facts that is necessary for us to prove the above Proposition A.2.\\
\textbf{Fact A:}$\ $ Let $f:M\rightarrow \mathbb{R}$ be bounded and continuous function and $f_t=\mathcal H_tf.$ Assume $\gamma:(0,1)\rightarrow M$ is a $f_t-$gradient curve which is also a constant-speed shortest curve. We  have :\\
\indent$(i)\quad$ $f_{t_1}(x)\ls f_{t_0}(y)+\frac{|xy|^2}{2(t_1-t_0)}$ for any $t_1>t_0>0$ and $x,y\in M$;\\
\indent$(ii)\quad$ $f_{t_1}(\gamma(t_1))= f_{t_0}(\gamma(t_0))+\frac{|\gamma(t_0)\gamma(t_1)|^2}{2(t_1-t_0)}$;\\
\indent$(iii)\quad$ $\nabla f_t=\gamma^+$ and $|\nabla f_t|=\frac{|\gamma(t_0)\gamma(t_1)|}{t_1-t_0}=|\gamma(0)\gamma(1)|$.

The following result is a modification of the proposition 2.2 in [Pet2], where we replace the condition curvature $\gs K$ by the condition $Ric(M)\gs (n-1)K.$
\begin{prop}
Let $M$ be an $n$-dimensional Alexandrov space with Ricci curvature
$\gs (n-1)K$. $f:M\rightarrow \mathbb{R}$ be bounded and continuous
function and $f_t=\mathcal H_tf.$ Assume $\gamma:(0,1)\rightarrow M$
is a $f_t-$gradient curve which is also a constant-speed shortest curve.
Suppose that the bilinear form $Hess_{\gamma(t)}f_t$ is defined for
almost all $t\in(0,1).$

 Then \begin{align*}
 h_T'&\ls-h^2_T,\\
 h_V'
&\ls -(n-1)K|\gamma(0)\gamma(1)|^2-\frac{h_V^2}{n-1}\end{align*} in the sense of distributions, where $$h_T(t)\overset{def}{=}Hess_{\gamma(t)}f_t\big(\frac{\gamma^+}{|\gamma^+|}, \frac{\gamma^+}{|\gamma^+|}\big)$$
and $h_V$ is the trace of $Hess_{\gamma(t)}f_t$ in the vertical  space $L_{\gamma(t)}$, i.e.,
$$h_V(t)\overset{def}{=}Trace_L Hess_{\gamma(t)}f_t.$$
\end{prop}

\begin{proof}
Since the bilinear form $Hess_{\gamma(t)}f_t$ is defined for almost
all $t\in(0,1)$, we know from [Pet1] that all $T_{\gamma(t)}$,
$t\in(0,1)$, are isometric to $n$-dimensional Euclidean space. In
particular, all $L_{\gamma(t)}$, $t\in(0,1)$, are isometric to
$\mathbb{R}^{n-1}$.

 Take two points $0<t_0<t_1<1,$
we may assume that $Hess_{\gamma(t)}f_t$ is defined at $t_0$ and $t_1$.

Denote by the direction $\xi_t=\gamma^+(t)/|\gamma^+(t)|$, $t\in(0,1)$. Then we have \begin{equation*}\begin{split}f_{t_0}\big(\gamma(t_0+s)\big)=&f_{t_0}\big(\gamma(t_0)\big)+s\cdot \langle\nabla f_{t_0},\gamma^+(t_0)\rangle\\&+\frac{s^2}{2}\cdot Hess_{\gamma(t_0)}f_{t_0}(\xi_{t_0},\xi_{t_0})\cdot|\gamma^+(t_0)|^2+o(s^2)\end{split}\end{equation*} and\begin{equation*}\begin{split}
f_{t_1}\big(\gamma(t_1+ls)\big)=&f_{t_1}\big(\gamma(t_1)\big)+ls\cdot \langle\nabla f_{t_1},\gamma^+(t_1)\rangle\\ &+\frac{(ls)^2}{2}\cdot Hess_{\gamma(t_1)}f_{t_1}(\xi_{t_1},\xi_{t_1})\cdot|\gamma^+(t_1)|^2+o(s^2)\end{split}\end{equation*}
for any $l\gs0$.
Combining these and  the Fact A, we get $$l^2\cdot h_T(t_1)-h_T(t_0)\ls \frac{(l-1)^2}{t_1-t_0}$$ for any $l\gs0$.

Thus, by choosing $ l=\big(1-(t_1-t_0)h_T(t_1)\big)^{-1}$ (when
$t_1-t_0$ suffices small, $1-(t_1-t_0)h_T(t_1)$ is positive), we get
$$\frac{h_T(t_1)-h_T(t_0)}{t_1-t_0}\ls-h_T(t_1)\cdot h_T(t_0).$$
That is, $$h_T'\ls-h^2_T.$$

Fix arbitrary $\epsilon>0$. By our definition of  Ricci curvature $\geqslant (n-1)K$ along $\gamma$, there exists a continuous function family $\{g_{\gamma(t)}\}_{0<t<1}\in\mathcal F$  such that \begin{equation}\oint_{\Lambda_{\gamma(t)}}g_{\gamma(t)}\gs K-\epsilon,\qquad \forall t\in(0,1).\end{equation}
We may assume $t_1-t_0$ so small that we can use equation (1.1) for some isometry $T: \Lambda_{\gamma(t_1)}\rightarrow\Lambda_{\gamma(t_0)}$ and some sequence $\{s_j\}\in \mathcal S$.

Given any direction $\eta\in\Lambda_{\gamma(t_1)},$ by setting
$\sigma_0(s)=\exp_{\gamma(t_0)}(s\cdot T\eta)$ and
$\sigma_1(s)=\exp_{\gamma(t_1)}(s\eta)$, we know from (1.1) that
\begin{equation}\begin{split}|\sigma_0(s_j)\ &\sigma_1(ls_j)|^2\ls |\gamma(t_0)\gamma(t_1)|^2\\&+\Big((l-1)^2-\frac{(g_{\gamma(t_1)}-\epsilon)\cdot|\gamma(t_0)\gamma(t_1)|^2}{3}(l^2+l+1)\Big)\cdot s_j^2+o(s_j^2).\end{split}\end{equation}

Note that \begin{equation}\begin{split} f_{t_0}(\sigma_0(s))&=f_{t_0}(\sigma_0(0))+\frac{s^2}{2}\cdot Hess_{\gamma(t_0)}f_{t_0}(T\eta,T\eta)+o(s^2),\\f_{t_1}(\sigma_1(ls))&=f_{t_1}(\sigma_1(0))+\frac{(ls)^2}{2}\cdot Hess_{\gamma(t_1)}f_{t_1}(\eta,\eta)+o(s^2)\end{split}\end{equation} for any $l\gs0$.
By combining (A.3), (A.4) and  the Fact A, we get
\begin{equation}\begin{split}l^2\cdot Hess_{\gamma(t_1)}&f_{t_1}(\eta,\eta)-Hess_{\gamma(t_0)}f_{t_0}(T\eta,T\eta)\\&\ls \frac{(l-1)^2}{t_1-t_0}-(t_1-t_0)|\gamma^+|^2\cdot\big(g_{\gamma(t_1)}(\eta)-\epsilon\big)\cdot\frac{l^2+l+1}{3}\end{split}\end{equation}
for any $l\gs0.$
 Set $\tau=t_1-t_0$ and $G=|\gamma^+(t_1)|^2\cdot(g_{\gamma(t_1)}(\eta)-\epsilon)$. By choosing
  $$ l=\big(1/\tau+\tau G/6\big)\cdot\big(1/\tau-\tau G/3-Hess_{\gamma(t_1)}f_{t_1}(\eta,\eta)\big)^{-1}$$(when $\tau$ suffices small,
  $1/\tau-\tau G/3-Hess_{\gamma(t_1)}f_{t_1}(\eta,\eta)$ and $ l$ are positive), we get
\begin{equation}\begin{split}
\big(\frac{1}{\tau}-\tau G/3\big)&\cdot\big(Hess_{\gamma(t_1)}f_{t_1}(\eta,\eta)-Hess_{\gamma(t_0)}f_{t_0}(T\eta,T\eta)\big)\\&\ls-Hess_{\gamma(t_1)}f_{t_1}(\eta,\eta)\cdot Hess_{\gamma(t_0)}f_{t_0}(T\eta,T\eta)-G+\tau^2G^2/12.
\end{split}\end{equation}

Note the simple fact that for an bilinear form $\beta(a,a)$ on a
$m-$dimensional inner product space $V^m$, $${\rm
trace}_{V^m}\beta=\frac{m}{vol(S)}\int_{S}\beta(a,a)da,$$ where $S$
is the unit sphere of $V^m$ with canonical measure. By taking trace
for $Hess_{\gamma(t_0)}f_{t_0}$ ( and $Hess_{\gamma(t_1)}f_{t_1}$ )
in $L_{\gamma(t_0)}$ (and $L_{\gamma(t_1)}$, respectively), we get,
from (A.2) and (A.6), that
\begin{equation}\begin{split} \frac{h_V(t_1)-h_V(t_0)}{\tau}\ls &-\frac{1}{2(n-1)}\big( h_V^2(t_0)+h_V^2(t_1)\big)\\&-(n-1)(K-2\epsilon)|\gamma^+(t_1)|^2+o(1)\end{split}\end{equation}
when we fix $t_1$ and let $t_0\to t_1.$

On the other hand, by setting $l=1$ in (A.5) and taking trace, we have
$$\frac{h_V(t_1)-h_V(t_0)}{\tau}\ls -(n-1)(K-2\epsilon)|\gamma^+(t_1)|^2.$$
This and (A.7) tell us that $h_V$ is locally Lipschitz almost everywhere in (0,1).

By using (A.7), the arbitrariness of $\epsilon$
and Fact A (iii), we get $$h_V'\ls-(n-1)K|\nabla
f_t|^2-\frac{h_V^2}{n-1}.$$ Therefore, we have completed the proof
of  this proposition.\end{proof}

Now we can follow Petrunin's argument in [Pet2] to prove the above
Proposition A.2.
\begin{proof} [Proof of Proposition A.2]
Let $\mu_0,\mu_1\in\mathcal P(M,d,m)$ with compactly supported sets $spt(\mu_0), spt(\mu_1)$ and $\mu_t\in\mathcal P(M,d)$ be a geodesic path. We have $$spt(\mu_t)\subset \bigcup_{x\in spt(\mu_0),\ y\in spt(\mu_1)}\gamma_{x,y}\qquad \forall t\in[0,1],$$where $\gamma_{x,y}$ is any one geodesic path between $x$ and $y$. Thus we can choose a big enough ball $B$ such that $ spt(\mu_t)\subset B$ for all $t\in [0,1]$. We can find a negative constant $k$ such that  $M$ has curvature $\gs k$ in $B$.

As shown in [V, 7.22], there is a probability measure $\Pi$ on the
space of all geodesic paths in $M$ such that  if $\Gamma = spt
(\Pi)$ and $e_t:\Gamma\rightarrow M$ is \emph{evaluation map}
$e_t(\gamma)=\gamma(t)$ then $\mu_t = (e_t)_{\#}\Pi$. Let $\Gamma $
be equipped a metric $$|\gamma \
\gamma'|_{\Gamma}:=\max_{t\in[0,1]}|\gamma(t)\gamma'(t)|.$$

According to [V, 5.10], there are a pair of optimal price functions
$\phi$ and $\psi$ on $M$ such that $$\phi(y)-\psi(x)\ls
\frac{1}{2}|xy|^2 $$ for any $x,y\in M$ and equality holds for any
$(x,y)\in spt\big((e_0,e_1)_{\#}\Pi\big).$

By considering the Hamilton--Jacobi shifts $$\psi_t=\mathcal
H_t\psi\quad {\rm and}\quad \phi_t=\mathcal H_{1-t}(-\psi),$$
Petrunin in [Pet2]  proved that,  for any $t\in(0,1)$,  $\mu_t$ is
absolutely continuous and the evaluation map $e_t$ is bi-Lipschitz
(where the bi-Lipschitz constant depends on $k$). Hence for any
measure $\chi$ on $M$, there is uniquely determined one-parameter
family of ¡°pull-back¡± measures $\chi^*_t$ on $\Gamma$ such that
$\chi^*_t(E)=\chi(e_tE)$ for any Borel subset $E\subset\Gamma.$
(Refer to [Pet2] for details),

Fix the measure  $\wt\nu=vol^*_{t_0=1/2}$ on $\Gamma$. We write
$vol^*_t=e^{w_t}\cdot\wt\nu$ for some Borel function
$w_t:\Gamma\rightarrow\mathbb{R},$ since $e_t$ is bi-Lipschitz and
$vol^*_t$ is absolutely continuous with respect to $\wt\nu$ for any
$t\in(0,1)$.

In [Pet2], Petrunin proved that, for  $\Pi-$a.e.  $\gamma\in \Gamma$, \begin{equation}w_t=\int_{t_0}^t\frac{\partial w_s}{\partial s}ds\qquad
 {\rm a.e.}\ \ t\in (0,1)\end{equation} and \begin{equation}\frac{\partial w_t}{\partial t}=h_t\qquad {\rm a.e.}\ \ t\in(0,1)\end{equation} where \begin{equation}h_t(\gamma)=Trace Hess_{\gamma(t)} \phi_t.\end{equation}

Noting that  $h_t=h_T(t)+h_V(t)$, we set $$ w^{(1)}_t=\int_{t_0}^t h_T(s)ds,\qquad B_1(t)=\exp(w^{(1)}_t)$$
 and$$ w^{(2)}_t=\int_{t_0}^t h_V(s)ds,\qquad B_2(t)=\exp\big(\frac{w^{(2)}_t}{n-1}\big).$$
By applying Proposition A.4, we get
\begin{equation}\begin{split}& B_1(t)\gs (1-t)B_1(0)+tB_1(1),\\
&B_2(t)\gs (1-t)\beta_{1-t}^{1/(n-1)}B_2(0)+t\beta_{t}^{1/(n-1)}B_2(1),
\end{split}\end{equation}
where $$\beta_t=\beta^{\big((n-1)K,n\big)}_t(\gamma(0),\gamma(1)).$$
Setting $D(t)=\exp(w_t/n)$ and using H\"{o}lder inequality$$\big(a+b\big)^{1/n}\cdot(c+d)^{(n-1)/n}\gs a^{1/n}\cdot c^{(n-1)/n}+ b^{1/n}\cdot d^{(n-1)/n}\quad \forall a,b,c,d>0,$$ we have
\begin{equation}\begin{split}D(t)&=B_1^{1/n}\cdot B_2^{(n-1)/n}\\& \gs \Big((1-t)B_1(0)+tB_1(1)\Big)^{\frac{1}{n}}\cdot \Big((1-t)\beta_{1-t}^{\frac{1}{n-1}}B_2(0)+t\beta_{t}^{\frac{1}{n-1}}B_2(1)\Big)^{\frac{n-1}{n}}\\&\gs (1-t)\beta_{1-t}^{1/n}B_1(0)B_2(0)^{(n-1)/n}+t\beta_{t}^{1/n}B_1(1)B_2(1)^{(n-1)/n}\\
&=(1-t)\beta_{1-t}^{1/n}D(0)+t\beta_{t}^{1/n}D(1).\end{split}\end{equation}

Note that  Petrunin  in [Pet2] had represented  $ H_n(\mu_t|m)$ in
terms of  $w_t(\gamma)$ as following,
$$H_n(\mu_t|m)=-\int_\Gamma\exp(w_t(\gamma)/n)\cdot ad\Pi$$ for some
non-negative Borel function $a:\Gamma\rightarrow\mathbb{R}.$ The
combination of this  with (A.12) implies the desired inequality
(A.1) in the definition of $CD(n,(n-1)K)$. Therefore we have
completed the proof of Proposition A.2.
\end{proof}

\end{document}